\documentclass{article}

\usepackage{geometry}
\usepackage{hyperref}
\usepackage{indentfirst}

\usepackage{graphicx,wrapfig,lipsum}
\usepackage{color}
\usepackage{transparent}

\usepackage{amsthm}
\usepackage{amssymb}
\usepackage{amsfonts}
\usepackage{amsmath}
\usepackage{mathtools}
\usepackage{enumitem}

\usepackage{adjustbox}
\usepackage{url} 
\usepackage[T1]{fontenc}

\usepackage{float}

\usepackage{tikz} 
\usepackage{tikz-cd}
\usetikzlibrary{positioning}
\usetikzlibrary{arrows}

\usepackage{diagbox}

\allowdisplaybreaks

\usepackage{titlesec}
\titleformat{\subsection}[runin]
  {\normalfont\bfseries}
  {\thesubsection}
  {0.5em}
  {}
  [.]

\usepackage{titlesec}
\titleformat{\subsubsection}[runin]
  {\normalfont\bfseries}
  {\thesubsubsection}
  {0.5em}
  {}
  [.]

\usepackage{mathrsfs}

\usepackage{stmaryrd}
\usepackage{old-arrows}

\usepackage{leftindex}

\usepackage{mathdots}

\DeclareFontFamily{U}{mathx}{}
\DeclareFontShape{U}{mathx}{m}{n}{<-> mathx10}{}
\DeclareSymbolFont{mathx}{U}{mathx}{m}{n}
\DeclareMathAccent{\widehat}{0}{mathx}{"70}
\DeclareMathAccent{\widecheck}{0}{mathx}{"71}

\newcommand{\C}{\mathbb{C}}
\newcommand{\R}{\mathbb{R}}
\newcommand{\Z}{\mathbb{Z}}

\newcommand{\Ball}{\mathbb{B}}
\newcommand{\Hyperbolic}{\mathbb{H}}
\newcommand{\Teich}{\mathcal{T}}

\newcommand{\RNum}[1]{\lowercase\expandafter{\romannumeral #1\relax}}
\newcommand{\URNum}[1]{\uppercase\expandafter{\romannumeral #1\relax}}

\DeclareMathOperator{\arccosh}{cosh^{-1}}
\DeclareMathOperator{\arcsinh}{sinh^{-1}}

\theoremstyle{definition}
\newtheorem*{defx*}{Definition}
\newtheorem{definition}{Definition}[section]

\theoremstyle{plain}
\newtheorem{thmx}{Theorem}

\newtheorem{lemma}[definition]{Lemma}
\newtheorem{corollary}[definition]{Corollary}

\newtheorem{proposition}[definition]{Proposition}

\theoremstyle{remark}
\newtheorem{remark}[definition]{Remark}

\DeclareMathOperator{\Mod}{Mod}

\DeclareMathOperator{\Aut}{Aut}
\DeclareMathOperator{\Hom}{Hom}

\DeclareMathOperator{\sys}{sys}

\DeclareMathOperator{\diam}{diam}
\DeclareMathOperator{\Sing}{Sing}
\DeclareMathOperator{\Int}{\iota}

\DeclareMathOperator{\Kob}{Kob}

\newcommand{\Moduli}{\mathcal{M}}

\newcommand{\dhyp}{d_{\Hyperbolic^2}}
\newcommand{\dtei}[1]{d_{\Teich_{#1}}}
\newcommand{\dm}[1]{d_{\Moduli_{#1}}}

\newcommand{\qii}{\text{\textbf{q.i.i.}}}
\newcommand{\qiit}{\widetilde{\text{\textbf{q.i.i.}}}}
\newcommand{\qiet}{\widetilde{\text{\textbf{q.i.e.}}}}
\usepackage{listings}

\usepackage{algorithm}
 \usepackage{algpseudocode}
 \usepackage{algorithmicx}
 \algdef{SE}[DOWHILE]{Do}{doWhile}{\algorithmicdo}[1]{\algorithmicwhile\ #1}%

\algrenewcommand\algorithmicrequire{\textbf{Input:}}
\algrenewcommand\algorithmicensure{\textbf{Output:}}



\NewDocumentCommand\vvec{mg}{\overrightarrow{#1}}

\usepackage[backend=biber,style=alphabetic,sorting=nyt,maxbibnames=10,maxalphanames=10,url=true,doi=false,giveninits=true,backref=true]{biblatex}

\AtEveryBibitem{\clearlist{language}}
\renewbibmacro{in:}{}

\DefineBibliographyStrings{english}{%
  backrefpage = {},%
  backrefpages = {}%
}

\renewbibmacro*{finentry}{\iflistundef{pageref}{}{}\finentry}
\renewbibmacro*{pageref}{%
  \iflistundef{pageref}
    {}
    {
      \setunit{\adddot\addspace}\printtext{%
        \ifnumgreater{\value{pageref}}{1}
          {\bibstring{backrefpages}\ppspace}
          {\bibstring{backrefpage}\ppspace}%
        \printlist[pageref][-\value{listtotal}]{pageref}\adddot
      }
    }
}

\usepackage{xurl}

\title{Quasi-isometric embeddings for shrinking maps\\from surfaces into the moduli space}

\author{
  \Large{Yibo Zhang} \\
  \small{\emph{Institut Fourier, UMR 5582, Laboratoire de Mathématiques}} \\
  \small{\emph{Université Grenoble Alpes, CS 40700, 38058 Grenoble cedex 9, France}} \\
  \small{\emph{email:}} \tt{zhangyibo12342000@gmail.com}
}

\date{}

\begin{document}


\maketitle

\begin{abstract}
We investigate shrinking maps from a cusped hyperbolic surface into the moduli space of closed Riemann surfaces.
For such a map and its lift to the Teichm\"uller space, we consider whether they are quasi-isometric embeddings with respect to natural metrics like the Teichm\"uller distance and the intrinsic distance.
Under a mild condition, we prove that these properties are characterised solely by the map's monodromy.
These characterisations apply, in particular, to holomorphic maps.
\end{abstract}

\section{Introduction}

Let $B$ be a hyperbolic surface of type $(g,n)$, where $g\ge 2$ and $n\ge 0$.
Given a map $F:B\rightarrow \Moduli_h$ into the moduli space of Riemann surfaces of genus $h\ge 2$,
any lift of $F$ to the universal covers is a map $\widetilde{F}: \Hyperbolic^2 \rightarrow \Teich_h$
from the upper half plane $\Hyperbolic^2 \subset \C$ to the Teichm\"uller space $\Teich_h$.
The induced homomorphism $F_*:\pi_1(B,t) \rightarrow \Mod_h$
from the surface group to the mapping class group is
called the \emph{monodromy representation} of $F$
(see Subsection \ref{subsection::monodromy}).

The map $F$ is called \emph{differentiable} if the lift $\widetilde{F}$ is differentiable.

The map $F$ is called \emph{holomorphic} if the lift $\widetilde{F}$ is holomorphic.
Here, the Teichm\"uller space
is equipped with the complex structure induced by the Bers embedding $\Teich_h \hookrightarrow \C^{3h-3}$ (see \cite{Bers1970}, \cite{Maskit1970}).
In this case,
the image $F(B)$ is called a \emph{holomorphic curve} in $\Moduli_h$.
By the hyperbolicity, we have the inequality
\begin{equation*}
\Kob_{\Hyperbolic^2}(\widetilde{b},\widetilde{v}) \ge \Kob_{\Teich_h}(\widetilde{F}(\widetilde{b}),d\widetilde{F}(\widetilde{v})),
\end{equation*}
where $\Kob_{\Hyperbolic^2}$
and $\Kob_{\Teich_h}$
denote the Kobayashi norms on $\Hyperbolic^2$ and $\Teich_h$, respectively.
Recall that the \emph{Kobayashi pseudonorm} $\Kob_X$ on a complex manifold $X$ is defined as the largest pseudonorm such that
every holomorphic map from the unit disc $\Ball \subset \C$,
equipped with the Poincar\'e metric of curvature $-4$, into $X$ is non-expanding.

Let $d_B$ and $d_{\Hyperbolic^2}$ denote the hyperbolic distances on $B$ and $\Hyperbolic^2$, respectively.
Let $d_{\Moduli_h}$ be the metric on $\Moduli_h$ induced by the Teichm\"uller distance $d_{\Teich_h}$ on $\Teich_h$.
It is known that the Kobayashi distance on $\Hyperbolic^2$ satisfies
$d_{\Hyperbolic^2,\Kob} = (1/2) d_{\Hyperbolic^2}$
(cf. \cite[Proposition 2.3.4]{Abate}),
while on $\Teich_h$,
the Kobayashi distance and the Teichm\"uller distance coincide, i.e., $d_{\Teich_h,\Kob} = d_{\Teich_h}$ (see \cite[Theorem 3]{Royden1971}).
Consequently, the holomorphic map $F:(B,(1/2)d_B)\rightarrow (\Moduli_h,d_{\Moduli_h})$ is distance-decreasing.

We define a distance $d_{\widetilde{F}}$ on $\Hyperbolic^2$ as the metric induced by the pullback $\widetilde{F}^* \Kob_{\Teich_h}$
of the Kobayashi norm on $\Teich_h$.
Similarly, we define a distance $d_F$ on $B$ as the metric induced by the pullback $F^* \Kob_{\Moduli_h}$, where $\Kob_{\Moduli_h}$ is the quotient Finsler norm on the orbifold $\Moduli_h$.
The metric $d_F$ then coincides with the quotient of $d_{\widetilde{F}}$ on $B$.

\begin{remark}
The intrinsic and extrinsic distances between two image points $q_1=F(b_1)$ and $q_2=F(b_2)$ are given by $d_F(b_1,b_2)$ and $d_{\mathcal{M}_h}(q_1,q_2)$, respectively.
Specifically: the intrinsic distance $d_F$ is the infimum of the lengths of all admissible paths connecting $q_1$ and $q_2$ that lie entirely on the image surface $F(B)$.
Here, ``admissible'' means that the path respects the intrinsic differential structure of $F(B)$, without exploiting extrinsic shortcuts (e.g., jumping across self-intersections).
The extrinsic distance $d_{\mathcal{M}_h}(q_1,q_2)$
is the shortest distance between $q_1$ and $q_2$ within the ambient space $\mathcal{M}_h$
and the minimizing path need not lie on $F(B)$.
\end{remark}

\begin{remark}
Analogously, for points
$\widetilde{F}(\widetilde{b_1})$ and $\widetilde{F}(\widetilde{b_2})$ in $\Teich_h$,
the intrinsic and extrinsic distances are given by  $d_{\widetilde{F}}(\widetilde{b_1},\widetilde{b_2})$ and $\dtei{h}(\widetilde{F}(\widetilde{b_1}),\widetilde{F}(\widetilde{b_2}))$, respectively.
\end{remark}

Holomorphic curves in the moduli space do indeed exist.
A prominent example is the \emph{Teichm\"uller curve} (cf. \cite{McMullen2023}),
defined as the image of a locally isometric embedding $F:B\rightarrow \Moduli_h$.
In this setting, the map $F$ is automatically holomorphic (see \cite{Antonakoudis2017}).
The local isometric embedding implies an isometry
$(B,(1/2)d_B) \cong (B,d_F)$ on the surface,
which lifts to an isometry
$(\Hyperbolic^2, (1/2)d_{\Hyperbolic^2}) \cong (\Hyperbolic^2,d_{\widetilde{F}})$ on the universal cover.
In fact, the image $\widetilde{F}(\Hyperbolic^2)$ forms a complex geodesic with respect to the intrinsic Kobayashi norm.
By Teichm\"uller's uniqueness theorem (cf. \cite[Theorems 11.8 and 11.9]{FarbMargalit2011}), any two points in $\Teich_h$ are connected by a unique real geodesic.
It follows that the lift $\widetilde{F}: (\Hyperbolic^2, (1/2)d_{\Hyperbolic^2}) \rightarrow (\Teich_h,d_{\Teich})$ is an isometric embedding.

We will study the differentiable map $F: B \rightarrow \mathcal{M}_h$ and its lift $\widetilde{F}$ to the Teichmüller space by proving that if $F$ satisfies the following quasi-isometric properties:

\vspace{1em}
\begin{minipage}[t]{0.15\linewidth}

$\qii(F):$

\end{minipage}%
\begin{minipage}[t]{0.75\linewidth}

The identity map is a quasi-isometry between $(B,(1/2)d_B)$ and $(B,d_F)$, namely, there exists $(\lambda,\epsilon)$ with $\lambda \ge 1$ and $\epsilon \ge 0$, such that
\begin{equation*}
\lambda d_B(b_1,b_2) + \epsilon \ge d_F(b_1,b_2) \ge \frac{1}{\lambda} d_B(b_1,b_2) - \epsilon
\end{equation*}
for all $b_1,b_2 \in B$.

\end{minipage}

\vspace{1em}
\begin{minipage}[t]{0.15\linewidth}

$\qiit(F):$

\end{minipage}%
\begin{minipage}[t]{0.75\linewidth}

For some (and hence every) lift $\widetilde{F}$ of $F$, the identity map is a quasi-isometry between $(\Hyperbolic^2,(1/2)d_{\Hyperbolic^2})$ and $(\Hyperbolic^2,d_{\widetilde{F}})$, namely, there exists $(\lambda,\epsilon)$ with $\lambda \ge 1$ and $\epsilon \ge 0$, such that
\begin{equation*}
\lambda \dhyp(\widetilde{b_1},\widetilde{b_2}) + \epsilon
\ge
d_{\widetilde{F}}( \widetilde{b_1}, \widetilde{b_2} )
\ge \frac{1}{\lambda} \dhyp(\widetilde{b_1},\widetilde{b_2}) - \epsilon
\end{equation*}
for all $\widetilde{b_1}, \widetilde{b_2} \in \Hyperbolic^2$.

\end{minipage}

\vspace{1em}
\begin{minipage}[t]{0.15\linewidth}

$\qiet(F):$

\end{minipage}%
\begin{minipage}[t]{0.75\linewidth}

For some (and hence every) lift $\widetilde{F}$ of $F$, the map $\widetilde{F}: (\Hyperbolic^2,(1/2)d_{\Hyperbolic^2}) \rightarrow (\Teich_h, d_\Teich)$ is a quasi-isometric embedding, namely, there exists
$(\lambda,\epsilon)$ with $\lambda \ge 1$ and $\epsilon \ge 0$, such that
\begin{equation*}
\lambda \dhyp(\widetilde{b_1},\widetilde{b_2}) + \epsilon
\ge \dtei{h}( \widetilde{F}(\widetilde{b_1}), \widetilde{F}(\widetilde{b_2}) )
\ge \frac{1}{\lambda} \dhyp(\widetilde{b_1},\widetilde{b_2}) - \epsilon
\end{equation*}
for all $\widetilde{b_1}, \widetilde{b_2} \in \Hyperbolic^2$.

\end{minipage}
\vspace{1em}

The map $F$ is called a \emph{quasi-isometric immersion} if it satisfies property $\qii(F)$.
The lift $\widetilde{F}$ is called a \emph{quasi-isometric immersion} (resp. \emph{quasi-isometric embedding}) if $F$ satisfies $\qiit(F)$ (resp. $\qiet(F)$).
For every holomorphic map $F$, the following implications hold:
$\qiet(F) \Rightarrow \qiit(F) \Rightarrow \qii(F)$,
as established in Lemma \ref{lemma::rigidities-implications}.
Moreover, a Teichm\"uller curve $F(B)$ satisfies all three properties: $\qii(F)$, $\qiit(F)$ and $\qiet(F)$,
as noted prior to their definitions.

\subsection{Quasi-isometric embeddings for a holomorphic map}

The work \cite{Zhang2024holomorphic} established that the monodromy provides a complete characterisation for the first quasi-isometric property $\qii$. Specifically:

\begin{thmx}[Theorem A - 2 in \cite{Zhang2024holomorphic}]
\label{thmx::qii}
Let $F:B\rightarrow \Moduli_h$ be a holomorphic map, where $B$ is a hyperbolic surface of type $(g,n)$ with $g\ge 2$, $n\ge 0$ and $h\ge 2$. Then $F$ satisfies $\qii(F)$ if and only if all peripheral monodromies of $F$ are of infinite order.
\end{thmx}

We extend this principle, demonstrating that the monodromy also fully characterises the other quasi-isometric properties studied here.
Our first result gives an analogous characterisation for $\qiit$.

\begin{thmx}
\label{thmx::qiit}
Let $F:B\rightarrow \Moduli_h$ be a holomorphic map, where $B$ is a hyperbolic surface of type $(g,n)$ with $g\ge 2$, $n\ge 0$ and $h\ge 2$. Then $F$ satisfies $\qiit(F)$ if and only if all peripheral monodromies of $F$ are of infinite order.
\end{thmx}

Consider a group $G$ acting by isometries on a proper length space $(X,d_X)$.
This action induces a pseudometric on $G$ defined by
$d_{G \curvearrowright X}(g_1,g_2) \coloneqq d_X(g_1 \cdot x, g_2 \cdot x)$,
where $x$ is a point in $X$.
The resulting pseudometric space $(G,d_{G \curvearrowright X})$ is well-defined up to quasi-isometry, independent of the choice of the basepoint $x$.

Now, let $B = \Gamma \backslash \Hyperbolic^2$ be a hyperbolic surface. 
Fixing an isomorphism $\rho: \pi_1(B,t) \rightarrow \Gamma$ allows us to define an isometric action of $\pi_1(B,t)$ on $\Hyperbolic^2$ via the deck transformation group $\Gamma$.
The induced metric $d_{\pi_1(B,t) \curvearrowright \Hyperbolic^2}$ is independent of the choice of $\rho$.

\begin{thmx}
\label{thmx::qiet}
Let $F:B\rightarrow \Moduli_h$ be a holomorphic map, where $B$ is a hyperbolic surface of type $(g,n)$ with $g\ge 2$, $n\ge 0$ and $h\ge 2$. Then $F$ satisfies $\qiet(F)$ if and only if
the monodromy representation $F_*: (\pi_1(B,t),d_{\pi_1(B,t) \curvearrowright \Hyperbolic^2}) \rightarrow (\Mod_h, d_{\Mod_h \curvearrowright \Teich_h})$ is a quasi-isometric embedding.
\end{thmx}

\begin{remark}
The property $\qiet(F)$ implies that $F_*$ is injective; see Lemma \ref{lemma::translation_length}.
\end{remark}

\begin{remark}
The quasi-isometric embedding
$F_*: (\pi_1(B,t),d_{\pi_1(B,t) \curvearrowright \Hyperbolic^2}) \rightarrow (\Mod_h, d_{\Mod_h \curvearrowright \Teich_h})$
is independent of the hyperbolic structure on $B$.
Specifically, if $B'$ is any hyperbolic surface diffeomorphic to $B$,
then the diffeomorphism $B\rightarrow B'$ induces a new group action of $\pi_1(B,t)$ on $\Hyperbolic^2$ via the deck transformation group of $B'$
but the resulting metric on $\pi_1(B,t)$ is quasi-isometric to $d_{\pi_1(B,t) \curvearrowright \Hyperbolic^2}$.
\end{remark}

\begin{remark}
Consider the case where $n=0$.
Note that the metric $d_{\pi_1(B,t) \curvearrowright \Hyperbolic^2}$ is quasi-isometric to the word metric.
When $F_*$ is injective and the monodromy group $G \coloneqq F_*(\pi_1(B,t))$ is convex cocompact, then the induced map
$F_*: (\pi_1(B,t),d_{\pi_1(B,t) \curvearrowright \Hyperbolic^2}) \rightarrow (\Mod_h, d_{\Mod_h \curvearrowright \Teich_h})$
becomes a quasi-isometric embedding.
Following \cite{Farb-Mosher-2002}, the convex cocompactness of $G$ requires that 
its orbit in Teichm\"uller space is quasi-convex.
Consequently, for any fixed $x \in \Teich_h$, the orbit map
$G \ni g \mapsto g \cdot x \in \Teich_h$ is a quasi-isometric embedding from $G$ to $\Teich_h$.
An equivalent definition, proved in \cite{Kent-Leininger-2008} and \cite{Hamenstaedt-2005}, states that convex cocompactness holds 
precisely when for any point $y$ in the curve graph $\mathcal{C}_h$ of the genus-$h$ surface, the orbit map $G \ni g \mapsto g \cdot y \in \mathcal{C}_h$ is a quasi-isometrc embedding.
This condition is further equivalent to requiring that
$F_*: (\pi_1(B,t),d_{\pi_1(B,t) \curvearrowright \Hyperbolic^2}) \rightarrow (\Mod_h, d_{\Mod_h \curvearrowright \mathcal{C}_h})$
be a quasi-isometric embedding.
\end{remark}

\begin{remark}
The quasi-isometric embedding
$F_*: (\pi_1(B,t),d_{\pi_1(B,t) \curvearrowright \Hyperbolic^2}) \rightarrow (\Mod_h, d_{\Mod_h \curvearrowright \Teich_h})$
implies that $F_*(\pi_1(B,t))$ is essentially purely having positive translation length.
That is, every non-trivial, non-peripheral element of $\pi_1(B,t)$ maps to an element in $\Mod_h$ with positive translation length (see Lemma \ref{lemma::translation_length} (a)).
When $n=0$, this property of the monodromy group is weaker than being purely pseudo-Anosov.
\end{remark}

\subsection{Quasi-isometric embeddings for a shrinking map}

This paper also extends most of the preceding results to shrinking differentiable maps, beginning with the following definition.

\begin{definition}
Let $B$ be a hyperbolic surface.
A differentiable map $F:B\rightarrow \Moduli_h$ is called \emph{$\lambda_0$-shrinking} for some $\lambda_0 \ge 1$ if, 
for some (and hence every) lift $\widetilde{F}$ of $F$,
the inequality
\begin{equation*}
\Kob_{\Hyperbolic^2}(\widetilde{b},\widetilde{v})
\ge
\frac{1}{\lambda_0}
\Kob_{\Teich_h}(\widetilde{F}(\widetilde{b}),d\widetilde{F}(\widetilde{v}))
\end{equation*}
holds for every $\widetilde{b} \in \Hyperbolic^2$
and every tangent vector $\widetilde{v} \in T_{\widetilde{b}} \Hyperbolic^2$.
\end{definition}

A $\lambda_0$-shrinking differentiable map $F:B\rightarrow \Moduli_h$ satisfies the following distance inequalities for all $b_1,b_2 \in B$:
\begin{align*}
\frac{1}{2} d_B(b_1,b_2) &\ge
\frac{1}{\lambda_0} d_F(b_1,b_2) \ge
\frac{1}{\lambda_0} d_{\Moduli_h}(F(b_1),F(b_2)), \\
\frac{1}{2} d_{\Hyperbolic^2}(\widetilde{b_1},\widetilde{b_2}) &\ge
\frac{1}{\lambda_0} d_{\widetilde{F}}(\widetilde{b_1},\widetilde{b_2}) \ge
\frac{1}{\lambda_0} d_{\Teich_h}(\widetilde{F}(\widetilde{b_1}),\widetilde{F}(\widetilde{b_2})).
\end{align*}
Here, the map $\widetilde{F}$ is any lift of $F$ and $\widetilde{b_1}, \widetilde{b_2}$ are any lifts of $b_1, b_2$ to $\Hyperbolic^2$, respectively.
In particular, every holomorphic map $F:B\rightarrow \Moduli_h$ is a $1$-shrinking differentiable map.

For a differentiable map $F:B\rightarrow \Moduli_h$, a point $b\in B$ is a \emph{singular point} of $F$
if some (and hence every) lift $\widetilde{b}$ of $b$ is a singular point of $\widetilde{F}$.
This means that the differential $d \widetilde{F}_{\widetilde{b}}: T_{\widetilde{b}} \Hyperbolic^2 \rightarrow T_{\widetilde{F}(\widetilde{b})} \Teich_h$ has rank strictly smaller than $2$.
We denote the set of all singular points of $F$ by $\Sing(F) \subset B$.

Now, consider a non-constant holomorphic map.
In local coordinates,
the differential $d \widetilde{F}_{\widetilde{b}}$ is represented by a $\C$-valued $(3h-3)$-by-$1$ matrix.
Consequently, the (real) rank of the differential $d \widetilde{F}_{\widetilde{b}}$ is either $0$ or $2$.
It follows that a point $b$ is singular if and only if $d \widetilde{F}_{\widetilde{b}} = (0,\ldots,0)^t \in \C^{(3h-3) \times 1}$.
By the identity theorem in complex analysis, the singular set $\Sing(F)$ is discrete and has no accumulation point.

\begin{thmx}
\label{thmx::shrinking}
Let $F:B\rightarrow \Moduli_h$ be a $1$-shrinking differentiable map, where $B$ is a hyperbolic surface of type $(g,n)$ with $g\ge 2$, $n\ge 0$ and $h\ge 2$.

\begin{enumerate}[label=(\alph*).]
\item If every peripheral monodromy of $F$ has infinite order, then $F$ satisfies $\qii(F)$.

\item If every peripheral monodromy of $F$ has infinite order and
the singular set $\Sing(F)$ is discrete with no accumulation point,
then $F$ satisfies $\qiit(F)$.

\item The map $F$ satisfies $\qiet(F)$
if and only if
the monodromy representation
\begin{equation*}
F_*: (\pi_1(B,t),d_{\pi_1(B,t) \curvearrowright \Hyperbolic^2}) \rightarrow (\Mod_h, d_{\Mod_h \curvearrowright \Teich_h})
\end{equation*}
is a quasi-isometric embedding.
\end{enumerate}
\end{thmx}

\begin{remark}
Recall the implications $\qiet(F) \Rightarrow \qiit(F) \Rightarrow \qii(F)$.
Thus, the monodromy condition equivalent to $\qiet(F)$ in Theorem \ref{thmx::shrinking} (c) also implies that $F$ satisfies $\qiit(F)$.
However, the singular set $\Sing(F)$ may fail to satisfy the discreteness condition required in Theorem \ref{thmx::shrinking} (b).
\end{remark}

\begin{remark}
The $1$-shrinking assumption
in Theorem \ref{thmx::shrinking}
is technically necessary.
The first, though not the only, obstruction
to extending these results to $\lambda_0$-shrinking maps with $\lambda_0 > 1$ is that
our methods do not anymore yield a bi-Lipschitz relation between the height of a point $b$ in a hyperbolic cusp region and the length of a curve on $F(b)$ supporting the peripheral monodromy,
when that peripheral monodromy has infinite order; see Proposition~\ref{proposition::l_epsilon}.
\end{remark}

\begin{proof}[Proof of Theorem \ref{thmx::qiet}]
This follows immediately from Theorem \ref{thmx::shrinking} (c).
\end{proof}

\subsection*{Outline}

Section \ref{section::preliminary} compiles the preliminary results.
Subsection \ref{subsection::hyperbolic_geodesics} provides a tractable estimate for the length of a geodesic segment within a hyperbolic cusp region.
Subsection \ref{subsection::teichmuller distance} provides estimates for the Teichmüller distance between a point and its image under a multi-twist.
Subsection \ref{subsection::rigidities} proves of the implications $\qiet \Rightarrow \qiit \Rightarrow \qii$.
Subsection \ref{subsection::monodromy} presents an explicit definition of the monodromy.

In Section \ref{section::peripheral}, we study the restriction of a shrinking differentiable map $F: B \rightarrow \Moduli_h$ to a cusp region $U \subset B$.
When the peripheral monodromy has infinite order, it is a root of a multi-twist. Proposition \ref{proposition::l_epsilon} in Subsection \ref{subsection::peripheral_infinite_order} establishes a coarse relation between the height of $b \in U$ and the lengths of curves on $F(b)$ supporting this multi-twist;
this relation is bi-Lipschitz when $F$ is $1$-shrinking.
We then decompose a path $\beta \subset B$ into segments
that either lie entirely within a cusp region or entirely outside all cusp regions. The lengths of the images of the segments within a cusp region are analysed in Subsection \ref{subsection::peripheral_infinite_order}.

The analysis for paths in the compact region $B_{\text{cp}} \subset B$ relies on the hypothesis concerning the singular set $\Sing(F)$, as detailed in Subsection \ref{subsection::compact_region}.
Subsection \ref{subsection::qii} presents the proof of Theorem \ref{thmx::shrinking} (a) and a short proof of Theorem \ref{thmx::qii}.
Subsection \ref{subsection::qiit} contains the proofs of Theorem \ref{thmx::shrinking} (b) and Theorem \ref{thmx::qiit}.
Subsection \ref{subsection::qiet} is devoted to the proof of Theorem \ref{thmx::shrinking} (c)

\subsection*{Acknowledgements}

This work was initiated during the author’s PhD at the Institut Fourier in Grenoble and completed during a visit to SIMIS in Shanghai.
The author is grateful to his advisors, Louis Funar and Greg McShane, for their valuable insights and constructive comments. He also thanks Xiaolong Hans Han for discussions and hospitality, and Veronica Pasquarella for helpful conversations.
Special gratitude goes to Mengmeng Xu for her meticulous reading of the manuscript.

\tableofcontents

\section{Preliminary results}
\label{section::preliminary}

We fix non-negative integers $g$, $n$ and $h$ with $2g-2+n > 0$ and $h \ge 2$ once and for all.

\subsection{Hyperbolic geodesics}
\label{subsection::hyperbolic_geodesics}

Given a cusped hyperbolic surface $B$, a cusp region is a neighbourhood of a cusp point bounded by a horocycle.
It is called \emph{standard} if
it has area $2$ and is bounded by a horocycle
of length $2$.
The cusp region $U \subset B$ can be expressed as the disjoint union
\[
U=\bigsqcup\limits_{0<\epsilon\le  l_B(\partial U)} H_{U,\epsilon}
\]
where $H_{U,\epsilon}$ is the horocycle of length $\epsilon$.
We define the \emph{height function} 
$\epsilon_U:U\rightarrow \R_{>0}$
by the condition that
$b\in H_{U,\epsilon_{U}(b)}\subset U$.

Consider a path $\beta\subset U$ joining $b_L$ and $b_R$.
The \emph{rounded winding number} $w_\beta \in \Z_{\ge 0}$ of $\beta$ is the (unoriented) winding number of $\beta$ around the cusp, rounded to the nearest integer.
When both $b_L$ and $b_R$ lie on $\partial U$, the number $w_\beta$ equals the number of self-intersections of $\beta$.

\begin{lemma}
\label{lemma::hyperbolic_length_cusp}
There exist constants $0<T_1\le T_2$ such that
for every geodesic segment $\beta \subset U$ joining $b_L$ to $b_R$ with $l_U(\beta) \ge 5$, the following holds:
Let $\epsilon_L=\epsilon_U(b_L)$,
$\epsilon_R=\epsilon_U(b_R)$
and let $w_\beta$ be the rounded winding number of $\beta$. Then
\begin{equation*}
T_1 \cdot \max\left\{ 
\left| \log{\frac{\epsilon_L}{\epsilon_R}} \right|,
\log{(w_\beta \epsilon_L)},
\log{(w_\beta \epsilon_R)}
\right\}
\le l_U(\beta) \le
T_2 \cdot \max\left\{ 
\left| \log{\frac{\epsilon_L}{\epsilon_R}} \right|,
\log{(w_\beta \epsilon_L)},
\log{(w_\beta \epsilon_R)}
\right\}.
\end{equation*}
\end{lemma}

\begin{remark}
When $\omega_{\beta} = 0$, the second and third terms inside the maximum evaluate to $-\infty$.
The maximum is therefore well-defined, as it is given by the first term, which is always positive.
\end{remark}

\begin{proof}
The hyperbolic distance (cf. \cite[Theorem 1.2.6]{Katok-1992}) between $x_1+iy_1$ and $x_2+iy_2$ is
\begin{equation*}
\dhyp(x_1+iy_1,x_2+iy_2) =
\arccosh \left(
1 + \frac{(x_2-x_1)^2+(y_2-y_1)^2}{2y_1y_2}
\right).
\end{equation*}
We represent $U = \{z=x+iy \in \Hyperbolic^2 \mid y \ge 1\} / (z \mapsto z+2)$.
Lift $\beta$ to $\widetilde{\beta} \subset \Hyperbolic^2$ joining $x_1+iy_1$ to $x_2+iy_2$,
where we assume that $x_1 \le x_2$ without loss of generality.
Define $\omega_\R = (x_2-x_1)/2$ and thus $\omega_\beta = \lfloor \omega_\R \rfloor$.
Define $\epsilon_L = 2/y_1$ and $\epsilon_R = 2/y_2$
and let $\Delta = \max\left\{ 
\left| \log{\frac{\epsilon_L}{\epsilon_R}} \right|,
\log{(w_\beta \epsilon_L)},
\log{(w_\beta \epsilon_R)}
\right\}$.

\textbf{Lower bound:}
Observe that
\begin{equation*}
l_U(\beta) =
\dhyp(x_1+iy_1,x_2+iy_2)
\ge
\dhyp(\R+iy_1,\R+iy_2)
=
\left| \log{\frac{y_1}{y_2}} \right|
=
\left| \log{\frac{\epsilon_L}{\epsilon_R}} \right|
\end{equation*}
and
\begin{align*}
l_U(\beta)
&=
\dhyp(x_1+iy_1,x_2+iy_2)
\ge
\dhyp(x_1+iy_1,x_2+i\R)
=
\frac{1}{2} \dhyp(x_1+iy_1,(2x_2-x_1)+iy_1)
\\
&=
\frac{1}{2} \arccosh
\left(
1 + \frac{(2x_2-2x_1)^2}{2y_1^2}
\right)
=
\frac{1}{2} \arccosh
\left(
1 + 2\cdot (\omega_\R \epsilon_L)^2
\right)
\\
&=
\frac{1}{2}
\log \left(
1 + 2 (\omega_\R \epsilon_L)^2
+
\sqrt{ \left( 1 + 2(\omega_\R\epsilon_L)^2 \right)^2 - 1 }
\right)
\ge
\frac{1}{2} \log (\omega_\R \epsilon_L)^2
= \log (\omega_\R \epsilon_L)
\end{align*}
and by symmetry, we get $l_U(\beta) \ge \log (\omega_\R\epsilon_R)$.

\textbf{Upper bound:}
We observe that
\begin{equation*}
l_U(\beta) \le
\dhyp(x_1+iy_1,x_1+iy_2)+\dhyp(x_1+iy_2,x_2+iy_2)
=
\left| \log \frac{\epsilon_L}{\epsilon_R} \right|
+
\frac{1}{2} \arccosh \left( 1 + \frac{1}{2} (\omega_\R \epsilon_R)^2 \right).
\end{equation*}
If $\left| \log \frac{\epsilon_L}{\epsilon_R} \right| \ge
\frac{1}{2} \arccosh \left( 1 + \frac{1}{2} (\omega_\R \epsilon_R)^2 \right)$,
then $l_U(\beta) \le 2 \left| \log \frac{\epsilon_L}{\epsilon_R} \right|$.
From now one, we assume that
$\left| \log \frac{\epsilon_L}{\epsilon_R} \right| \le
\frac{1}{2} \arccosh \left( 1 + \frac{1}{2} (\omega_\R \epsilon_R)^2 \right)$.
As $l_U(\beta) \ge 5$, we obtain that $5 \le \arccosh \left( 1 + \frac{1}{2} (\omega_\R \epsilon_R)^2 \right)$.
Therefore $\omega_\R \epsilon_R \ge \sqrt{2 (\cosh{5}-1)} \approx 12.1004 \ge 12 \ge 1$
and $\log (\omega_\R \epsilon_R) \ge 1$.
Hence, we get
\begin{align*}
\frac{1}{2} \arccosh \left( 1 + \frac{1}{2} (\omega_\R \epsilon_R)^2 \right)
&=
\frac{1}{2}
\log \left(
1 + \frac{1}{2}(\omega_\R\epsilon_R)^2 +
\sqrt{ \left(
1 + \frac{1}{2}(\omega_\R\epsilon_R)^2
\right)^2 - 1}
\right) \\
&\le
\frac{1}{2} \log{ \frac{3+\sqrt{5}}{2} (\omega_\R \epsilon_R)^2 }
=
\frac{1}{2} \log \frac{3+\sqrt{5}}{2} + \log (\omega_\R \epsilon_R)
\le 2 \log (\omega_\R \epsilon_R).
\end{align*}

\textbf{Conclusion:}
Combining these estimates gives:
\begin{equation*}
\max \left\{
\left| \log{\frac{\epsilon_L}{\epsilon_R}} \right|,
\log(\omega_\R \epsilon_L),
\log(\omega_\R \epsilon_R)
\right\} \le l_U(\beta)
\le
3 \cdot
\max \left\{
\left| \log{\frac{\epsilon_L}{\epsilon_R}} \right|,
\log(\omega_\R \epsilon_L),
\log(\omega_\R \epsilon_R)
\right\}.
\end{equation*}
The result follows by considering the cases $\omega_\beta \ge 1$ and $\omega_\beta = 0$ separately.

If $\omega_\beta \ge 1$, then $\omega_\beta \le \omega_\R \le 2\omega_\beta$.
Hence, we get
$\log (\omega_\R \epsilon_L) \ge \log (\omega_\beta \epsilon_L)$,
$\log (\omega_\R \epsilon_R) \ge \log (\omega_\beta \epsilon_R)$
and
$\log (\omega_\R \epsilon_L) \le \log (\omega_\beta \epsilon_L) + \log(2)$,
$\log (\omega_\R \epsilon_R) \le \log (\omega_\beta \epsilon_R) + \log(2)$.
Hence, $\Delta \le l_U(\beta) \le 3\Delta + 3 \log(2)$.
As $\Delta \le \log(2)$ would imply the contradiction $l_U(\beta) \le 6 \log{2} \approx 4.159 < 5$,
we obtain that $\Delta \ge \log{2}$ and $l_U(\beta) \le 6\Delta$.

If $\omega_{\beta}=0$, then $0 \le \omega_\R < 1$.
Thus $\Delta \le l_U(\beta) \le 3\cdot \max\{ \left| \log \frac{\epsilon_L}{\epsilon_R} \right|, \log (2) \} \le 3 \Delta + 3 \log(2)$.
Similarly, as $\Delta \le \log(2)$ would imply a contradiction, we obtain that $l_U(\beta) \le 6 \Delta$.
\end{proof}

\subsection{Teichm\"uller distance}
\label{subsection::teichmuller distance}

The Teichm\"uller space
$\Teich_h$ parameterizes all complex structures on a closed, oriented smooth
surface $\Sigma_h$ of genus $h\ge 2$, up to isotopy.
Elements in $\Teich_h$ are written as $[X,f_X]$, where $X$ is a Riemann surface and $f_X:\Sigma_h \rightarrow X$ is an orientation preserving diffeomorphism.
Two pairs $[X,f_X]$ and $[Y,f_Y]$ are equal if $f_Y \circ f_X^{-1} : X \rightarrow Y$ is isotopic to a biholomorphism.

The mapping class group $\Mod_h$ consists of all orientation-preserving diffeomorphisms of $\Sigma_h$, up to isotopy.
This group acts properly discontinuously on the Teichm\"uller space $\Teich_h$ via the action
$[\phi] \cdot [X,f_X] = [X, f_X \circ \phi^{-1}]$.
A \emph{multi-curve} $\boldsymbol{\alpha}=\{\alpha_1,\ldots,\alpha_m\}$ on $\Sigma_h$ is a collection of disjoint, non-isotopic simple closed curves.
A \emph{multi-twist} along $\boldsymbol{\alpha}$ is a product of Dehn twists $T=T_{\alpha_1}^{r_1}\circ \cdots \circ T_{\alpha_m}^{r_m}$,
where each $r_i \in \Z \setminus \{0\}$.

The \emph{Teichm\"uller distance} between two points $[X,f_X], [Y,f_Y]\in\Teich_{g,n}$ is given by
\begin{equation*}
\dtei{h}([X,f_X],[Y,f_Y])=\frac{1}{2} \log \inf_{\phi}\{K(\phi)\},
\end{equation*}
where the infimum is taken over all quasiconformal diffeomorphisms $\phi:X\rightarrow Y$ homotopic to $f_Y\circ f_X^{-1}$ and $K(\phi)\ge 1$ denotes the dilatation of $\phi$.

The \emph{geodesic length function} $L_\gamma([X,f_X])$ assigns to each closed curve $\gamma \subset \Sigma_h$ the length of the unique geodesic homotopic to $f_X(\gamma)$ on the hyperbolic representative $X$ of $[X,f_X] \in \Teich_h$.
The geodesic length
function is distorted by a factor of at most $\exp (2 \dtei{h})$, namely,
$L_\gamma(x) / L_\gamma(y) \le \exp (2 \dtei{h}(x,y) ) $,
due to Wolpert \cite{Wolpert-1979}.
The \emph{translation length} of a mapping class $\phi \in \Mod_h$ is defined by $\tau(\phi) = \inf_{x \in \Teich_h} \dtei{h}(x, \phi \cdot x).$

\begin{lemma}
\label{lemma::l_le_teich}
Let $\widetilde{q} \in \Teich_h$ and $T=T_{\alpha_1}^{r_1} \circ \cdots \circ T_{\alpha_m}^{r_m} \in \Mod_h$ be a multi-twist, with each $r_i \neq 0$.
Suppose that $d_{\Teich}(\widetilde{q}, T \cdot \widetilde{q}) \le N_2$, for some constant $N_2>0$.
Then
\begin{equation*}
L_{\alpha_i}(\widetilde{q}) \le C_1(N_2) \cdot \dtei{h}(\widetilde{q}, T \cdot \widetilde{q}),
\end{equation*}
where the constant $C_1(N_2)$ depends only on $N_2$.
\end{lemma}

\begin{proof}
This estimate is given by \cite[Proposition 3.11]{Zhang2024holomorphic}.
\end{proof}

\begin{lemma}
\label{lemma::teich_le_log(rl)}
Let $\widetilde{q} \in \Teich_h$ and 
consider a multi-twist $T=T_{\alpha_1}^{r_1} \circ \cdots \circ T_{\alpha_m}^{r_m} \in \Mod_h$, where each $r_i \neq 0$.
Define $\widetilde{q'} = T \cdot \widetilde{q}$
and suppose that for some constant $N_1 \ge 0$, we have
$L_{\alpha_i}(\widetilde{q}) \le N_1$ for every $i$.
Then the following estimates hold for the Teichmüller distance $\dtei{h}(\widetilde{q},\widetilde{q'})$:
\begin{enumerate}[label=(\alph*).]
\item If $\max_i \{ |r_i| \cdot L_{\alpha_i}(\widetilde{q}) \} \ge \pi^2$, then
\begin{equation*}
\dtei{h}(\widetilde{q},\widetilde{q'})
\ge
\frac{1}{4} \cdot \log \max_i \left\{ |r_i| \cdot L_{\alpha_i}(\widetilde{q}) \right\}.
\end{equation*}

\item There exists a constant $C_2(N_1)$, depending only on $N_1$, such that if
$\max_i \{ |r_i| \cdot L_{\alpha_i}(\widetilde{q}) \} \ge C_2(N_1)$, then
\begin{equation*}
\dtei{h}(\widetilde{q},\widetilde{q'})
\le
2 \cdot \log \max_i \left\{ |r_i| \cdot L_{\alpha_i}(\widetilde{q}) \right\}.
\end{equation*}
\end{enumerate}
\end{lemma}

\begin{proof}
By Theorem 1 of \cite{Matsuzaki2003}, we obtain the following estimates:
\begin{align*}
\dtei{h}(\widetilde{q}, \widetilde{q'}) \ge
&
\frac{1}{2} \cdot \log \max_i
\left(
\left(
\frac{(2 |r_i| - 1) \cdot L_{\alpha_i}(\widetilde{q})}{\pi}
\right)^2
+ 1
\right)^{1/2},
\\
\dtei{h}(\widetilde{q}, \widetilde{q'}) \le
&
\frac{1}{2} \cdot \log \max_i
\left(
\left(
\left(
\frac{|r_i| \cdot L_{\alpha_i}(\widetilde{q})}{2 \theta_i}
\right)^2
+ 1
\right)^{1/2}
+
\frac{|r_i| \cdot L_{\alpha_i}(\widetilde{q})}{2 \theta_i}
\right)^2,
\end{align*}
where
$\theta_i \coloneqq \pi - 2 \arctan\left( \sinh \left( \frac{L_{\alpha_i}(\widetilde{q})}{2} \right) \right)$.

\textbf{Lower bounds:}
Since $|r_i| \ge 1$ for all $i$, we have $2|r_i| - 1 \ge |r_i|$. This implies
\begin{equation*}
\dtei{h}(\widetilde{q}, \widetilde{q'})
\ge
\frac{1}{2} \cdot \frac{1}{2} \cdot 2 \cdot
\log \max_i
\frac{(2 |r_i| - 1) \cdot L_{\alpha_i}(\widetilde{q})}{\pi}
\ge
\frac{1}{2} \cdot
\log \max_i
\frac{|r_i| \cdot L_{\alpha_i}(\widetilde{q})}{\pi}.
\end{equation*}
If $\max_i \{ |r_i| \cdot L_{\alpha_i}(\widetilde{q}) \} \ge \pi^2$, then
\begin{equation*}
\dtei{h}(\widetilde{q}, \widetilde{q'})
\ge
\frac{1}{2} \cdot
\log\left(
\sqrt{
\max_i
\{|r_i| \cdot L_{\alpha_i}(\widetilde{q})\}
}
\cdot
\sqrt{
\frac{
\max_i
\{|r_i| \cdot L_{\alpha_i}(\widetilde{q})\}
}{\pi^2}
}
\right)
\ge
\frac{1}{4}
\log \max_i
\{
|r_i| \cdot L_{\alpha_i}(\widetilde{q})
\}.
\end{equation*}

\textbf{Upper bound:}
Define the constant
\begin{equation*}
C_2(N_1) \coloneqq
\max\left\{
2 \sqrt{2} \left(
\pi - 2 \arctan \left(
\sinh \frac{N_1}{2}
\right)
\right),
\left( \sqrt{2} + 1 \right)
\cdot
\frac{1}{2 \sqrt{2} \left(
\pi - 2 \arctan \left(
\sinh \frac{N_1}{2}
\right)
\right)}
\right\}.
\end{equation*}
Assume that $\max_i \left\{ |r_i| \cdot L_{\alpha_i}(\widetilde{q}) \right\} \ge C_2(N_1)$ and thus
\begin{equation*}
\left(
\frac{ \max_i \left\{ |r_i| \cdot L_{\alpha_i}(\widetilde{q}) \right\} }{ 2\left( \pi - 2 \arctan \left( \sinh \frac{N'_1}{2} \right) \right) }
\right)^2 
\ge
\left(
\frac{ \max_i \left\{ |r_i| \cdot L_{\alpha_i}(\widetilde{q}) \right\} }{ \frac{C_2(N_1)}{\sqrt{2}} }
\right)^2 
\ge 2.
\end{equation*}
Since $\theta_i \ge \pi - 2 \arctan\left( \sinh \frac{N_1}{2} \right)$, we have
\begin{align*}
\dtei{h}(\widetilde{q},\widetilde{q'})
& \le
\log \left(
\left(
\left(
\frac{ \max_i \left\{ |r_i| \cdot L_{\alpha_i}(\widetilde{q}) \right\} }{ 2\left( \pi - 2 \arctan \left( \sinh \frac{N_1}{2} \right) \right) }
\right)^2
+ 1
\right)^{1/2}
+
\frac{ \max_i \left\{ |r_i| \cdot L_{\alpha_i}(\widetilde{q}) \right\} }{ 2\left( \pi - 2 \arctan \left( \sinh \frac{N_1}{2} \right) \right) }
\right) \\
& \le
\log \left(
\left(\sqrt{2} + 1 \right)
\cdot
\frac{1}{2\left( \pi - 2 \arctan \left( \sinh \frac{N_1}{2} \right) \right)}
\cdot
\max_i \left\{ |r_i| \cdot L_{\alpha_i}(\widetilde{q}) \right\}
\right) \\
& \le
\log \left(
C_2(N_1) \cdot
\max_i
\{
|r_i| \cdot L_{\alpha_i}(\widetilde{q})
\}
\right)
\le 2 \cdot
\log
\max_i
\{
|r_i| \cdot L_{\alpha_i}(\widetilde{q})
\}.
\end{align*}
\end{proof}

\subsection{Relations between quasi-isometric properties}
\label{subsection::rigidities}

\begin{lemma}
\label{lemma::rigidities-implications}
Let $F:B\rightarrow \Moduli_h$ be a $\lambda_0$-shrinking differentiable map.
Then the following implications hold:
$\qiet(F) \Rightarrow \qiit(F) \Rightarrow \qii(F)$.
\end{lemma}

\begin{proof}
Suppose that $F$ satisfies $\qiet(F)$, namely, there exists $(\lambda,\epsilon)$ such that
\begin{equation*}
\frac{\lambda_0}{2} \dhyp(\widetilde{b_1},\widetilde{b_2})
\ge
\dtei{h}(\widetilde{F}(\widetilde{b_1}),\widetilde{F}(\widetilde{b_2}))
\ge \frac{1}{\lambda} \cdot \dhyp(\widetilde{b_1},\widetilde{b_2})
- \epsilon
\end{equation*}
for some (and hence every) lift $\widetilde{F}$ of $F$, for all $\widetilde{b_1},\widetilde{b_2} \in \Hyperbolic^2$.
As $d_{\widetilde{F}}(\cdot,\cdot) \ge \dtei{h}(\widetilde{F}(\cdot),\widetilde{F}(\cdot))$, we get
\begin{equation*}
\frac{\lambda_0}{2} \dhyp(\widetilde{b_1},\widetilde{b_2})
\ge
d_{\widetilde{F}}(\widetilde{b_1},\widetilde{b_2})
\ge
\dtei{h}(\widetilde{F}(\widetilde{b_1}),\widetilde{F}(\widetilde{b_2}))
\ge \frac{1}{\lambda} \cdot \dhyp(\widetilde{b_1},\widetilde{b_2})
- \epsilon
\end{equation*}
and hence $F$ satisfies $\qiit(F)$.

Suppose that $F$ satisfies $\qiit(F)$, namely, there exists $(\lambda,\epsilon)$ such that
\begin{equation*}
\frac{\lambda_0}{2} \dhyp(\widetilde{b_1},\widetilde{b_2})
\ge
d_{\widetilde{F}}(\widetilde{b_1},\widetilde{b_2})
\ge \frac{1}{\lambda} \cdot \dhyp(\widetilde{b_1},\widetilde{b_2})
- \epsilon.
\end{equation*}
Consider $b_1, b_2 \in B$.
There exist $\widetilde{b_1}, \widetilde{b_2} \in \Hyperbolic^2$, which are lifts of $b_1, b_2$, respectively, such that $d_F(b_1,b_2) = d_{\widetilde{F}}(\widetilde{b_1},\widetilde{b_2})$.
Thus, we get
\begin{equation*}
\frac{\lambda_0}{2} d_B(b_1,b_2)
\ge
d_F(b_1,b_2)
=
d_{\widetilde{F}}(\widetilde{b_1},\widetilde{b_2})
\ge
\frac{1}{\lambda} d_{\Hyperbolic^2}(\widetilde{b_1},\widetilde{b_2}) - \epsilon
\ge
\frac{1}{\lambda} d_B(b_1,b_2) - \epsilon
\end{equation*}
and hence $F$ satisfies $\qii(F)$.
\end{proof}

\subsection{Monodromy representations}
\label{subsection::monodromy}

Consider a differentiable map $F:B\rightarrow\mathcal{M}_h$
and the lift $\widetilde{F}:\Hyperbolic^2\rightarrow \Teich_h$.
We obtain the induced \emph{monodromy representation} $F_*$ as follows.

Suppose that $B=\Gamma \backslash \Hyperbolic^2$ where $\Gamma\le \Aut(\Hyperbolic^2)$ is a lattice.
The map $\widetilde{F}:\Hyperbolic^2\rightarrow\Teich_h$ induces a group homomorphism $F_{\Gamma}:\Gamma\rightarrow \Mod_h$ such that
$\widetilde{F}\circ \phi = F_{\Gamma}(\phi)\circ \widetilde{F}$,
for every $\phi\in\Gamma$.
The homomorphism $F_{\Gamma}$ is not necessarily unique
when $F(t)$ has non-identity automorphisms (i.e. $F(t)$ is symmetric), for some $t\in B$.
Fixing a base point $t\in B$ and lifting it to some $\widetilde{t}\in\Hyperbolic^2$, we obtain a group isomorphism $\rho_{t,\widetilde{t}}:\pi_1(B,t)\rightarrow \Gamma$
such that every loop $\gamma\subset B$ based at $t$ is lifted to the path in $\Hyperbolic^2$ joining $\widetilde{t}$ to $\rho_{t,\widetilde{t}}([\gamma]) \cdot \widetilde{t}$.

We call $F_*\coloneqq F_{\Gamma}\circ \rho_{t,\widetilde{t}} \in \Hom(\pi_1(B,t), \Mod_h)$
a \emph{monodromy representation} of $F$ with the lift $\widetilde{F}$.
Different choices of $F_{\Gamma}$ and $\rho_{t,\widetilde{t}}$ change the monodromy representation by conjugacies.
The image $F_*([\gamma])$ along a based loop $\gamma \subset B$ that goes once or several times around a cusp, clockwise or counterclockwise, is called a \emph{peripheral monodromy} of
the cusp.

In general, we consider a representation $\rho:\pi_1(B,t) \rightarrow \Mod_h$ of the surface group in $\Mod_h$.

\begin{lemma}
\label{lemma::translation_length}
Suppose that the representation
$\rho:(\pi_1(B,t),d_{\pi_1(B,t) \curvearrowright \Hyperbolic^2}) \rightarrow (\Mod_h, d_{\Mod_h \curvearrowright \Teich_h})$ 
is a quasi-isometric embedding.
Then, we have the following for the image.

\begin{enumerate}[label=(\alph*).] 
\item For every non-trivial, non-peripheral element $[\gamma] \in \pi_1(B,t)$, the image $\rho([\gamma]) \in \Mod_h$ has positive translation length.

\item For every peripheral element $[\gamma] \in \pi_1(B,t)$, the image $\rho([\gamma]) \in \Mod_h$ has zero translation length and infinite order. In this case,
there exists $\mu \in \Z_{\ge 1}$ such that $\phi^\mu$ is a multi-twist.
\end{enumerate}
\end{lemma}

\begin{proof}
(a) Assume that $[\gamma] \in \pi_1(B,t)$ is a non-trivial, non-peripheral element such that $\phi \coloneqq \rho([\gamma])$ has zero translation length.
Recall the classification of mapping classes due to Bers \cite{Bers1978},
the translation length of $\phi$ is attained if and only if $\phi$ has finite order $\mu \ge 1$; otherwise, there exists $\mu \in \Z_{\ge 1}$ such that $\phi^\mu = \prod_i T_{\alpha_i}^{r_i}$ is a multi-twist.

Suppose that $[\gamma]$ is represented by a closed geodesic $\beta \subset B$.
Let $b \in \beta$ be arbitrary.
Consider the axis $\widetilde{\beta}$ of the hyperbolic isometry $[\gamma]$, which is a lift of $\beta$.
Take some lift $\widetilde{b} \in \widetilde{\beta}$ of $b$.
Consider the metric $d_{\pi_1(B,t) \curvearrowright \Hyperbolic^2}$
using the orbit $\pi_1(B,t) \cdot \widetilde{b}$.
We also fix a point $\widetilde{q} \in \Teich_h$
and consider the metric $d_{\Mod_h \curvearrowright \Teich_h}$
using the orbit $\Mod_h \cdot \widetilde{q}$.
For every $k \in \Z_{\ge 1}$,
take $\widetilde{b_k} = [\gamma]^{k\cdot \mu} \cdot \widetilde{b}$ and
$\widetilde{q_k} = \phi^{k\cdot \mu}\cdot \widetilde{q}$.
We get
\begin{equation*}
d_{\pi_1(B,t) \curvearrowright \Hyperbolic^2}(1, [r]^{k \cdot \mu})
=
\dhyp(\widetilde{b}, \widetilde{b_k}) = k\cdot \mu \cdot l_B(\beta)
\end{equation*}
and
\begin{equation*}
d_{\Mod_h \curvearrowright \Teich_h}(1, \phi^{k \cdot \mu})
=
\dtei{h}(\widetilde{q}, \widetilde{q_k})
=
\dtei{h}(\widetilde{q}, (\phi^\mu)^k \cdot \widetilde{q}).
\end{equation*}
If $\phi$ has finite order $\mu$, then $\dtei{h}(\widetilde{q},\widetilde{q_k}) = 0$.
If $\phi$ has infinite order, then
$\dtei{h}(\widetilde{q},\widetilde{q_k}) =
\dtei{h}(\widetilde{q}, \prod_i T_{\alpha}^{k\cdot r_i} \cdot \widetilde{q})$
which,
by Lemma \ref{lemma::teich_le_log(rl)},
is asymptotically approximated by $\log(k)$.
In both cases, the representation $\rho$ cannot be a quasi-isometric embedding, which is a contradiction.

(b) Let $[\gamma] \in \pi_1(B,t)$ be the peripheral element around some cusp region $U \subset B$.
Taking $b \in \partial U$ and some lift $\widetilde{b}\in \Hyperbolic^2$ of $b$,
we consider the metric $d_{\pi_1(B,t) \curvearrowright \Hyperbolic^2}$ using the orbit $\pi_1(B,t) \cdot \widetilde{b}$.
For every $k \in \Z_{\ge 1}$, there exists the geodesic segment $\beta_k \subset U$ joining $b$ to itself representing $[\gamma]^k$, such that $l_B(\beta_k)$ 
is asymptotically approximated by $\log(k)$,
so does $d_{\pi_1(B,t) \curvearrowright \Hyperbolic^2}(1, [\gamma]^k)$.

Assume that $\rho([\gamma])$ has finite order.
Then $d_{\Mod_h \curvearrowright \Teich_h}(1, \rho([\gamma])^k)$ is bounded, for $k \in \Z_{\ge 1}$, which is a contradiction.
Assume that $\rho([\gamma])$ has positive translation length.
Then $d_{\Mod_h \curvearrowright \Teich_h}(1, \rho([\gamma])^k )$
is asymptotically linearly approximated by $k$,
which is also a contradiction.
\end{proof}

As a consequnce, the quasi-isometric embedding implies that the peripheral monodromies of any two cusps are not disjoint along any geodesic segment between the boundaries of their cusp regions, as the following.

\begin{lemma}
\label{lemma::peripherals_intersecting}
Suppose that the representation
$\rho:(\pi_1(B,t),d_{\pi_1(B,t) \curvearrowright \Hyperbolic^2}) \rightarrow (\Mod_h, d_{\Mod_h \curvearrowright \Teich_h})$ 
is a quasi-isometric embedding.
Let $U_1, U_2 \subset B$ be two cusp regions, which could be the same.
Let $\kappa \subset B$ be a geodesic segment perpendicularly joining $\partial U_1$ to $\partial U_2$.
Set $t_0 = \partial U_1 \cap \kappa$ and take an arbitrary path $\gamma$ joining $t$ to $t_0$;
see Figure \ref{figure::kappa}.
The loop along $\gamma \cup \partial U_1$ based at $t$ that goes once around $U_1$ clockwise is denoted by $\gamma_1$.
The loop along $\gamma \cup \kappa \cup \partial U_2$ based at $t$ that goes once around $U_2$ clockwise is denoted by $\gamma_2$.
Then the monodromies $\phi_1 \coloneqq \rho([\gamma_1])$ and $\phi_2 \coloneqq \rho([\gamma_2])$ are roots of multi-twists supported by intersecting multi-curves.
\end{lemma}

\begin{figure}[htb]
\centering
\def\svgwidth{.9\textwidth}
\input{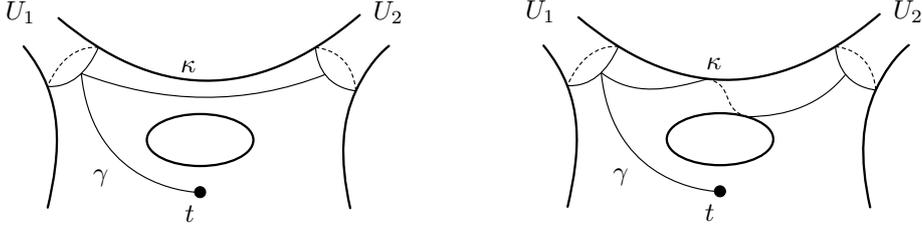}
  \caption{Different selections of the geodesic segment $\kappa$ joining the boundaries of two given cusp regions $U_1$ and $U_2$, each providing a pair of peripheral monodromies $(\phi_1,\phi_2)$.}
  \label{figure::kappa}
\end{figure}

\begin{proof}
By Lemma \ref{lemma::translation_length} (b), there exist $\mu_1, \mu_2 \in \Z_{\ge 1}$ such that
$\phi_1^{\mu_1}$ is a multi-twist along the multi-curve $\boldsymbol{\alpha}_1$
and 
$\phi_2^{\mu_2}$ is a multi-twist along the multi-curve $\boldsymbol{\alpha}_2$.
Assume that $\boldsymbol{\alpha}_1 \cup \boldsymbol{\alpha}_2$ is itself a multi-curve, i.e., a set of disjoint simple closed curves. 
Then, the element $[\gamma_1]^{2 \mu_1} \cdot [\gamma_2]^{2 \mu_2} \in \pi_1(B,t)$ is non-peripheral which is represented by a multi-twist along $\boldsymbol{\alpha}_1 \cup \boldsymbol{\alpha}_2$, which contradicts with Lemma \ref{lemma::translation_length} (a).
\end{proof}

\section{Peripheral monodromy}
\label{section::peripheral}

\subsection{Peripheral monodromies of finite order}
\label{subsection::peripheral_finite_order}

We only consider the holomorphic map $F:B\rightarrow \Moduli_h$ in this subsection.

\begin{proposition}[Proposition 2.5 in \cite{Zhang2024holomorphic}]
\label{proposition::torsion-peripheral-qii}
Let $F:B\rightarrow \Moduli_h$ be a holomorphic map.
If some peripheral monodromy of $F$ has finite order, then $F$ does not satisfy $\qii(F)$.
\end{proposition}

\begin{corollary}
\label{corollary::torsion-peripheral-qiit}
Let $F:B\rightarrow \Moduli_h$ be a holomorphic map.
If some peripheral monodromy of $F$ has finite order, then $F$ does not satisfy $\qiit(F)$.
\end{corollary}

\begin{proof}
This comes from Proposition \ref{proposition::torsion-peripheral-qii} and Lemma \ref{lemma::rigidities-implications}.
\end{proof}

\subsection{Peripheral monodromies of infinite order}
\label{subsection::peripheral_infinite_order}

Consider a $\lambda_0$-shrinking differentiable map $F:B\rightarrow \Moduli_h$
with the lift $\widetilde{F}:\Hyperbolic^2 \rightarrow \Teich_h$.
Suppose that all peripheral monodromies of the $\lambda_0$-shrinking map $F:B\rightarrow \Moduli_h$ are of infinite order.
We consider the restriction of $F$ to some cusp region $U\subset B$ with $l_B(\partial U)=\epsilon_0 \le 2$.

Fix $t \in \partial U$ with the lift $\widetilde{t}$
and consider the monodromy representation $F_*$ of $F$.
Suppose that the peripheral monodromy
$\phi = F_*(\partial U) \in \Mod_h$
is of infinite order.
Therefore, there exists a positive integer $\mu=\mu(\phi)$ bounded above by a constant determined by $h$ such that $\phi^\mu$ is a multi-twist,
namely $\phi^\mu = T_{\alpha_1}^{r_1}\circ \cdots \circ T_{\alpha_m}^{r_m} \eqqcolon T$ with each $r_i \in \Z_{\neq 0}$.
Every $\alpha_i$ also induces a multi-valued length map on $U$,
defined by
\begin{equation*}
\mathcal{L}_{\alpha_i}:U
\ni b
\mapsto
\left[
L_{\alpha_i}\left( \widetilde{F}(\widetilde{b})
\right),
L_{\alpha_i}\left( \phi \cdot \widetilde{F}(\widetilde{b})
\right),
\ldots,
L_{\alpha_i}\left( \phi^{\mu-1}\cdot
\widetilde{F}(\widetilde{b})
\right)
\right]
\in
\R_{>0}^\mu / (\Z / \mu \Z),
\end{equation*}
where $\beta \subset B$ is a path joining $t$ to $b$ with the lift $\widetilde{\beta}$ joining $\widetilde{t}$ to $\widetilde{b}$,
where the group $\Z / \mu \Z$ acts by cyclic permutation.
It is straightforward to check that $\mathcal{L}_{\alpha_i}$ is well-defined as
$\phi \circ \widetilde{F} = \widetilde{F} \circ [\partial U]$
and
$L_{\alpha_i} \circ \phi^\mu = L_{\alpha_i}$.

\begin{proposition}
\label{proposition::l_epsilon}
There exists a constant $K^+(F,\epsilon_0) > 0$,
depending only on $F$ and $\epsilon_0$,
such that for all $i=1,\ldots,m$, $j=1,\ldots,\mu$ and $b \in U$, we have
\begin{equation*}
\frac{1}{K^+(F,\epsilon_0)} \left( \epsilon_U(b) \right)^{\lambda_0}
\le
(\mathcal{L}_{\alpha_i}(b))_j
\le 
K^+(F,\epsilon_0) \cdot \epsilon_U(b).
\end{equation*}
\end{proposition}

\begin{proof}
Given $b\in U$, let $\widetilde{q} \in \Teich_h$ be a lift of $q = F(b) \in \Moduli_h$.
As $\dtei{h}(\widetilde{q}, \phi^\mu \cdot \widetilde{q}) \le \frac{\lambda_0}{2} \cdot \epsilon_U(b) \mu \le 
\frac{\lambda_0}{2} \cdot \epsilon_0 \mu$,
by Lemma \ref{lemma::l_le_teich}, we get
\begin{equation*}
\big(\mathcal{L}_{\alpha_i}(b)\big)_j
\le
C_1\left(\frac{\lambda_0}{2} \cdot \epsilon_0 \mu\right) \cdot
\frac{\lambda_0}{2} \cdot
\epsilon_U(b) \mu.
\end{equation*}

On the other hand, suppose that $\beta$ is a geodesic segment joining $b$ to $b_0 \in \partial U$ perpendicularly to the boundary.
There exists a lift of $F(\beta)$ joining $\widetilde{q}$ to some lift $\widetilde{q_0}$ of $F(b_0)$.
By Wolpert's inequality, we get
\begin{equation*}
\frac{1}{2} \log{ \frac{\epsilon_U(b_0)}{\epsilon_U(b)} }
= \frac{1}{2} d_U(b_0,b)
\ge
\frac{1}{\lambda_0} \dtei{h}(\widetilde{F}(\widetilde{b_0}), \widetilde{F}(\widetilde{b}))
\ge
\frac{1}{2\lambda_0}
\log{ \frac{L_{\alpha_i}(\widetilde{q_0})}{L_{\alpha_i}(\widetilde{q})}},
\end{equation*}
for every $i=1,\ldots,m$.
Hence, we obtain 
\begin{equation*}
\big( \mathcal{L}_{\alpha_i}(b) \big)_j \ge
\frac{1}{\epsilon_0^{\lambda_0}} \min_{d \in \partial U} \min_{i',j'} \big( \mathcal{L}_{\alpha_{i'}}(d) \big)_j \cdot \left(\epsilon_U(b)\right)^{\lambda_0}.
\end{equation*}
\end{proof}

\subsubsection{Quasi-isometric embedding of the cusp region}

If $F$ is a holomorphic map, then its restriction to the cusp region $U$ is a quasi-isometric embedding
$\left. F \right|_{U} : (U,d_U) \rightarrow (\Moduli_h,\dm{h})$,
where $d_U = \left. d_B \right|_U$
(see \cite[Theorem 3.12]{Zhang2024holomorphic}).
In general, we have the following.

\begin{proposition}
\label{proposition::qie-cusp}
If the differentiable map F is 1-shrinking,
then there exist constants $K_1(F,\epsilon_0)>0$ and $K_2(F,\epsilon_0)$,
depending only on $F$ and $\epsilon_0$,
such that for all $b_1,b_2 \in U$, we have
\begin{equation*}
\frac{1}{2} d_U(b_1,b_2) \ge
\dm{h}(F(b_1),F(b_2)) \ge
\frac{1}{K_1(F,\epsilon_0)} d_U(b_1,b_2) - K_2(F,\epsilon_0).
\end{equation*}
\end{proposition}

The following proof, though similar to that of Theorem 3.12 in \cite{Zhang2024holomorphic}, is included for completeness.

\begin{proof}[Proof of Proposition \ref{proposition::qie-cusp}]
We first consider the special case where one point is lying on the boundary $\partial U$.
Let $b_{\max}\in\partial U$ and $b\in H_{U,\epsilon} \subset U$ be arbitrary with $\epsilon\le \epsilon_0 \le 2$.
Take the lift $\widetilde{b_{\max}} \in \Hyperbolic^2$ of $b_{\max}$
and the lift
$\widetilde{b} \in \Hyperbolic^2$ of $b$ such that $\dm{h}(F(b_{\max}), F(b))=
\dtei{h}(\phi_{\max} \cdot \widetilde{F}(\widetilde{b_{\max}}), \phi \cdot \widetilde{F}(\widetilde{b}))$,
for some $\phi_{\max}, \phi \in \Mod_h$.
For convenience, we set $q_{\max}=F(b_{\max})$, $q=F(b)$, $\widetilde{q_{\max}}=\widetilde{F}(\widetilde{b_{\max}})$ and $\widetilde{q}=\widetilde{F}(\widetilde{b})$.
By Wolpert's inequality, Lemma \ref{lemma::l_le_teich}
with $\dtei{h}(\widetilde{q}, T \cdot \widetilde{q}) \le (1/2) \epsilon \mu \le \mu$
and the triangle inequality in $(U,d_U)$, we get
\begin{align*}
\dm{h}(q,q_{\max})
&=\dtei{h}(\phi^{-1}\circ \phi_{\max} \cdot \widetilde{q_{\max}},\widetilde{q})
\ge
\frac{1}{2}\log{
  \frac{L_{\alpha_1}(\phi^{-1}\circ \phi_{\max} \cdot \widetilde{q_{\max}})}{L_{\alpha_1}(\widetilde{q})}
}\\
&\ge
\frac{1}{2}\log{
  \frac{\sys(\widetilde{q_{\max}})}{C_1(\mu) \dtei{h}(\widetilde{q}, T\cdot \widetilde{q})}
}
\ge
\frac{1}{2}\log{
  \frac{\sys(q_{\max})}{C_1(\mu) \mu \epsilon / 2}
}
=
\frac{1}{2}\Big\{\log{
  \frac{\sys(q_{\max})}{\epsilon_0 C_1(\mu) \mu / 2}
}
+
\log{\frac{\epsilon_0}{\epsilon}}\Big\}\\
&\ge
\frac{1}{2}d_U(b,b_{\max})-K'_2
\end{align*}
where
$\frac{1}{2} d_U(b,b_{\max}) \le \frac{1}{2}\left( \log \frac{\epsilon_0}{\epsilon} + \epsilon_0 \right)$
and
$K'_2=\epsilon_0/2-(1/2)\log{\sys(q_{\max})}+(1/2)\log{(\epsilon_0 C_1(\mu) \mu/2)}$.

In general, let $b_1,b_2\in U$ be arbitrary.
Set $q_1=F(b_1)$, $q_2=F(b_2)$ and
take the corresponding horocycles
$H_{U,\epsilon_1}\ni p_1$, $H_{U,\epsilon_2}\ni p_2$.
Using the above inequality and triangle inequalities in both $(U,d_U)$ and $(\mathcal{M}_h,\dm{h})$, we conclude that
\begin{align*}
\frac{1}{2}d_U(p_1,p_2)
&\ge
\dm{h}(q_1,q_2) \\
&\ge
|\dm{h}(q_1,q_{\max})-\dm{h}(q_2,q_{\max})|
\ge
\frac{1}{2}|d_U(p_1,p_{\max})-d_U(p_2,p_{\max})|-K'_2\\
&\ge
\frac{1}{2} \max\Big\{
  \Big(\log{\frac{\epsilon_0}{\varepsilon_1}}-\epsilon_0\Big)-
  \Big(\log{\frac{\epsilon_0}{\varepsilon_2}}+\epsilon_0\Big),
  \Big(\log{\frac{\epsilon_0}{\varepsilon_2}}-\epsilon_0\Big)-
  \Big(\log{\frac{\epsilon_0}{\varepsilon_1}}+\epsilon_0\Big)
\Big\}-K'_2\\
&=
\frac{1}{2} \max\Big\{
  \log{\frac{\varepsilon_1}{\varepsilon_2}},
  \log{\frac{\varepsilon_2}{\varepsilon_1}}
\Big\}-\epsilon_0-K'_2\\
&\ge
\frac{1}{2}(d_U(p_1,p_2)-\epsilon_0)-\epsilon_0-K'_2
=
\frac{1}{2}d_U(p_1,p_2)-\frac{3}{2} \epsilon_0 - K'_2.
\end{align*}
\end{proof}

\subsubsection{Quasi-isometric lengths along the path}

Let $\beta \subset U$ be a path joining $b_L$ to $b_R$.
Let $\widetilde{\beta} \subset \Hyperbolic^2$ be a lift of $\beta$ joining $\widetilde{b_L}$ to $\widetilde{b_R}$.
Set $q_L = F(b_L)$,
$q_R = F(b_R)$,
$\widetilde{q_L} = \widetilde{F}(\widetilde{b_L})$ and $\widetilde{q_R} = \widetilde{F}(\widetilde{b_R})$.
Consider the hyperbolic geodesic $\beta' \subset B$ rel isotopic to $\beta$.

We prove that the lengths $l_U(\beta)$ and $l_{\Moduli_h}(F(\beta))$,
as well as the distance
$d_{\Teich_h}(\widetilde{q_L},\widetilde{q_R})$,
are quasi-isometric.

\begin{proposition}
\label{proposition::cusp-quasi-isom}
If the differentiable map $F$ is $1$-shrinking, then there exist constants $K_3(F,\epsilon_0) > 0$ and $K_4(F,\epsilon_0)$,
depending only on $F$ and $\epsilon_0$,
such that
\begin{equation*}
\frac{1}{2} l_U(\beta) \ge l_{\Moduli_h} \big( F(\beta) \big)
\ge \dtei{h}( \widetilde{F}(\widetilde{b_L}), \widetilde{F}(\widetilde{b_R}) )
\ge \frac{1}{K_3(F,\epsilon_0)} l_U(\beta') - K_4(F,\epsilon_0).
\end{equation*}
\end{proposition}

\begin{proof}
Set $\epsilon_L=\epsilon_U(b_L)$ and $\epsilon_R=\epsilon_U(b_R)$.
Without loss of generality, we assume that $\epsilon_L \ge \epsilon_R$.
The first two desired inequalities are known, as $F$ is $1$-shrinking.
We also know that $(1/2) l_U(\beta') \ge \dtei{h}(\widetilde{q_L},\widetilde{q_R})$.
To prove the third desired inequality, it suffices to show that either $(1/2)l_U(\beta')$ is bounded from above
or, by Lemma \ref{lemma::hyperbolic_length_cusp},
we have $\dtei{h}(\widetilde{q_L},\widetilde{q_R})$ is quasi-bounded below by both of $\left|\log{ \frac{\epsilon_L}{\epsilon_R}}\right|$ and $\log (\omega_\beta \epsilon_L)$.

By Proposition \ref{proposition::qie-cusp}, we first get
\begin{align}
\dtei{h}(\widetilde{q_L}, \widetilde{q_R})
\ge
\dm{h}(q_L,q_R)
&\ge
\frac{1}{K_1(F,\epsilon_0)} d_U(b_1,b_2) - K_2(F,\epsilon_0) 
\nonumber \\
&\ge
\frac{1}{K_1(F,\epsilon_0)}
\left(
\left|\log{ \frac{\epsilon_L}{\epsilon_R}}\right| - \epsilon_0
\right)
- K_2(F,\epsilon_0).
\label{equation::dt-log(eL/eR)}
\end{align}

It remains to show that $\dtei{h}(\widetilde{q_L},\widetilde{q_R})$ is quasi-bounded below by $\log (\omega_\beta \epsilon_L)$,
so we further assume that
$\log (\omega_\beta \epsilon_L) \ge \left| \log \epsilon_L/\epsilon_R \right|$ and
$\omega_\beta \ge 1$.
Then, there exists $\widetilde{b_M} \in \Hyperbolic^2$,
which is another lift of $b_L$, such that
$\widetilde{q_M} = (\phi^\mu)^t \cdot \widetilde{q_L}$ and
$\dtei{h}(\widetilde{q_M},\widetilde{q_R}) \le \frac{1}{2}
\left(
\left|\log{ \frac{\epsilon_L}{\epsilon_R}}\right| + \mu\epsilon_0
\right)$,
where $\widetilde{q_M} = \widetilde{F}(\widetilde{b_M})$ and $t\in \Z$.
If $|t| \le 2$, then
$\omega_\beta \cdot \epsilon_L \le 3\mu \cdot \epsilon_0 \le 6\mu$
and thus
$l_U(\beta') \le
T_2 \cdot \log (\omega_\beta \epsilon_L) \le
T_2 \cdot \log (6 \mu)$,
where $T_2$ is the constant introduced in Lemma \ref{lemma::hyperbolic_length_cusp}.

Assume that $|t| \ge 3$.
Then $|t| \cdot \mu \ge (1/2) \omega_\beta$.
By Proposition \ref{proposition::l_epsilon}, we further conclude with $\max_i \{ |t| \cdot |r_i| \cdot L_{\alpha_i}(\widetilde{q_L}) \}
\ge \frac{1}{2\mu \cdot K^+(F,\epsilon_0)} \max_i \{|r_i|\} \cdot \omega_\beta \epsilon_L$.
Taking $T = (\phi^\mu)^t = \prod_i T_{\alpha_i}^{t\cdot r_i}$,
by Lemma \ref{lemma::teich_le_log(rl)} (a),
we obtain that
either $\max_i \{ |t| \cdot |r_i| \cdot L_{\alpha_i}(\widetilde{q_L}) \} \le \pi^2$, hence
\begin{equation*}
l_U(\beta')
\le
T_2 \cdot \log (\omega_\beta \epsilon_L)
\le
T_2 \cdot \log \left(
\pi^2 \cdot 2\mu \cdot K^+(F,\epsilon_0) \cdot \frac{1}{\max\{|r_i|\}}
\right)
\end{equation*}
or
\begin{align}
\dtei{h}(\widetilde{q_L}, \widetilde{q_R})
&\ge
\dtei{h}(\widetilde{q_L}, \widetilde{q_M}) -
\dtei{h}(\widetilde{q_M}, \widetilde{q_R})
\ge
\dtei{h}(\widetilde{q_L}, \widetilde{q_M}) -
\frac{1}{2} \left(
\left|\log{ \frac{\epsilon_L}{\epsilon_R}}\right| + \mu\epsilon_0 
\right)
\nonumber \\
&\ge \frac{1}{4} \log
\max_i \left\{
|t| \cdot |r_i| \cdot L_{\alpha_i}(\widetilde{q_L})
\right\} -
\frac{1}{2} \left(
\left|\log{ \frac{\epsilon_L}{\epsilon_R}}\right| + \mu\epsilon_0
\right)
\nonumber \\
&\ge \frac{1}{4} \log
\left(
\frac{1}{2\mu \cdot K^+(F,\epsilon_0)}
\max_i \left\{|r_i|\right\}
\cdot \omega_\beta \epsilon_L
\right) -
\frac{1}{2} \left(
\left|\log{ \frac{\epsilon_L}{\epsilon_R}}\right| + \mu\epsilon_0
\right).
\label{equation::dt-we-log(eL/eR)}
\end{align}

Write the inequality (\ref{equation::dt-we-log(eL/eR)}) as
$\dtei{h}(\widetilde{q_L},\widetilde{q_R}) \ge \frac{1}{K'_3} \cdot \log (\omega_\beta \epsilon_L) - 
\frac{1}{2} \left|\log{ \frac{\epsilon_L}{\epsilon_R}}\right| - K'_4$
with some constants $K'_3 > 0$ and $K'_4$.
If $\left|\log{ \frac{\epsilon_L}{\epsilon_R}}\right| \le (1/K'_3) \cdot \log (\omega_\beta \epsilon_L)$,
then $\dtei{h}(\widetilde{q_L},\widetilde{q_R}) \ge \frac{1}{2 K'_3} \log (\omega_\beta \epsilon_L) - K'_4$.
Otherwise, by the inequality (\ref{equation::dt-log(eL/eR)}), we get
\begin{equation*}
\dtei{h}(\widetilde{q_L},\widetilde{q_R})
\ge \frac{1}{K_1(F,\epsilon_0)}
\left(
\left|
\log \frac{\epsilon_L}{\epsilon_R}
\right|
- \epsilon_0
\right)
- K_2(F,\epsilon_0)
\ge
\frac{1}{K_1(F,\epsilon_0)}
\left(
\frac{1}{K'_3}
\log
(\omega_\beta \epsilon_L)
- \epsilon_0
\right)
- K_2(F,\epsilon_0).
\end{equation*}
\end{proof}

\subsubsection{Bi-Lipschitz lengths along the path joining boundary points}

Consider $\beta \subset U$ as before.

When the path $\beta\subset U$ is joining two boundary points $b_L,b_R\in \partial U$,
the lengths $l_U(\beta')$ and $l_{\Moduli_h}(F(\beta))$ 
further have a bi-Lipschitz relation,
unless $\beta$ is away from an explicit smaller cusp region.
More precisely, we have the following.

\begin{proposition}
\label{proposition::cusp-bi-lipschitz}
Suppose that $b_L\in \partial U$ and $b_R \in \partial U$.
If the differentiable map $F$ is $1$-shrinking,
then there exist constants $K_5(F,\epsilon_0) > 0$ and $0 < \epsilon^+(F,\epsilon_0) \le \epsilon_0$,
depending only on $F$ and $\epsilon_0$,
such that
either
\begin{equation*}
\beta \subset \bigsqcup_{\epsilon^+(F,\epsilon_0) \le \epsilon \le \epsilon_0} H_{U,\epsilon}
\quad
\text{or}
\quad
\frac{1}{2} l_U(\beta) \ge l_{\Moduli_h} \big( F(\beta) \big)
\ge \frac{1}{K_5(F,\epsilon_0)} l_U(\beta').
\end{equation*}
\end{proposition}

\begin{proof}
By Proposition \ref{proposition::cusp-quasi-isom},
there exist constants $K'_3(F,\epsilon_0)>0$ and $K'_4(F,\epsilon_0)$ depending only on $F$ and $\epsilon_0$, such that either $l_U(\beta') \le K'_4(F,\epsilon_0)$
or $(1/2) l_U(\beta) \ge l_{\Moduli_h}(F(\beta)) \ge (1/K'_3(F,\epsilon_0)) l_U(\beta')$.

Set $\epsilon_{\text{min}} \coloneqq \inf_{b \in \beta} \epsilon_U(b) \le \epsilon_0$
and suppose that $\epsilon_U(b_{\text{min}}) = \epsilon_{\text{min}}$ for some $b_{\text{min}} \in \beta$.
By Proposition \ref{proposition::qie-cusp}, we get
\begin{align*}
l_{\Moduli_h}(F(\beta))
\ge
d_{\Moduli_h}(F(b_L),F(b_{\text{min}}))
&\ge
\frac{1}{K_1(F,\epsilon_0)} d_U(b_L,b_{\text{min}}) - K_2(F,\epsilon_0) \\
&\ge
\frac{1}{K_1(F,\epsilon_0)}\left(
\log \frac{\epsilon_0}{\epsilon_{\text{min}}} - 1
\right)
- K_2(F,\epsilon_0).
\end{align*}
Thus, there exists a sufficiently small $\epsilon^+(F,\epsilon_0)$ such that
\begin{equation*}
\frac{1}{K_1(F,\epsilon_0)}
\left(
\log{\frac{\epsilon_0}{\epsilon}} - 1
\right)
- K_2(F,\epsilon_0)
> K'_4(F,\epsilon_0)
\end{equation*}
for every $0 < \epsilon < \epsilon^+(F,\epsilon_0)$.

If $\epsilon_{\text{min}} < \epsilon^+(F,\epsilon_0)$, then $l_{\Moduli_h}(F(\beta)) > K'_4(F,\epsilon_0)$.
Hence, either $l_{\Moduli_h}(F(\beta)) > K'_4(F,\epsilon_0) \ge l_U(\beta')$
or $l_{\Moduli_h}(F(\beta)) \ge (1/K'_3(F,\epsilon_0)) l_U(\beta')$.
We conclude that
either $\beta \subset \bigsqcup_{\epsilon^+(F,\epsilon_0) \le \epsilon \le \epsilon_0} H_{U,\epsilon}$
or
$l_{\Moduli_h}(F(\beta)) \ge \min\{1, 1/K'_3(F,\epsilon_0)\} \cdot l_U(\beta')$.
\end{proof}

\section{Conclusions}
\label{section::conclusion}

\subsection{Paths in the compact region}
\label{subsection::compact_region}

This subsection consider the restriction of a $1$-shrinking differentiable map $F:B\rightarrow \Moduli_h$ to some compact part $B_{\text{cp}}(\boldsymbol{\epsilon})$ of $B$ by removing a cusp region for each cusp point.
More precisely, we have the following.

\begin{proposition}
\label{proposition::bilipschitz-path-Bcp}
Let $F:B\rightarrow \Moduli_h$ be a $1$-shrinking differentiable map
such that the set of singular points $\Sing(F) \subset B$ is discrete and has no accumulation point.
Let $B_{\text{cp}}(\boldsymbol{\epsilon}) \subset B$ with $\boldsymbol{\epsilon} = (\epsilon_1,\ldots,\epsilon_n)$ 
be the subsurface obtained by removing,
for each $i=1,\ldots,n$,
the cusp region of the $i$-th cusp point bounded by a horecycle of length $0 < \epsilon_i \le 2$.
Then the following hold::

\begin{enumerate}[label=(\alph*).] 
\item The intersection $\Sing(F) \cap B_{\text{cp}}(\boldsymbol{\epsilon})$ is a finite set, denoted by $\{p_1,p_2,\ldots,p_k\}$.

\item There exists $r > 0$ such that the hyperbolic neighbourhoods $U(p_i,r)$ of the point $p_i$ are are pairwise disjoint for $i=1,2,\ldots,k$.

\item For \( r > 0 \) as in (b),
there exists constants $K_7(F,\boldsymbol{\epsilon},r)>0$
and $K_8(F,\boldsymbol{\epsilon},r)$,
depending only on $F$, $\boldsymbol{\epsilon}$ and $r$, such that
for every path \( \beta \subset B_{\text{cp}}(\boldsymbol{\epsilon}) \), we have
\begin{equation*}
\frac{1}{2} l_{B}(\beta)
\ge
l_{\Moduli_h}(F(\beta)) \ge 
\frac{1}{K_7(F,\boldsymbol{\epsilon})} l_{B}(\beta')
-
K_8(F,\boldsymbol{\epsilon}),
\end{equation*}
where $\beta'$ is the geodesic rel isotopic to $\beta$.
Moreover, for every path \( \beta \subset B_{\text{cp}}(\boldsymbol{\epsilon}) \) connecting two points in the closure \( \overline{B_{\text{cp}}(\boldsymbol{\epsilon}) \setminus \bigsqcup_i U(p_i, r)} \), we have
\begin{equation*}
\frac{1}{2} l_{B}(\beta)
\ge
l_{\Moduli_h}(F(\beta)) \ge 
\frac{1}{K_7(F,\boldsymbol{\epsilon})} l_{B}(\beta').
\end{equation*}
\end{enumerate}
\end{proposition}

\begin{proof}
The assertions (a) and (b) follow immediately.

Consider the path $\beta \subset B_{\text{cp}}(\boldsymbol{\epsilon})$.
Let $\beta''$ be obtained from $\beta$ by straightening each subarc that locates within some $U(p_i,r/3)$
and replacig each subarc that locates within some annulus $U(p_i,r) \setminus \overline{U(p_i,r/3)}$ with one shortest rel isotopic path within this annulus.
Let $\beta'$ be the geodesic rel isotopic to $\beta$.

Let $K_r$ be the closure 
$\overline{ B_{\text{cp}}(\boldsymbol{\epsilon}) \setminus \bigcup_{i=1}^k U(p_i,r/3) }$.
As $K_r$ is compact, we set
\begin{equation*}
K_7(F,\boldsymbol{\epsilon},r)^{-1} \coloneqq
\inf_{b\in K_r}
\inf_{v\in T_b^1 B}
\left\{
\frac{\Kob_{\Teich_h}(\widetilde{F}(\widetilde{b}), d\widetilde{F}(\widetilde{v})) }{ \Kob_{\Hyperbolic^2}(\widetilde{b},\widetilde{v})}
\right\}
> 0,
\end{equation*}
where the selection of the lift $(\widetilde{b},\widetilde{v})$ of $(b,v)$ does not change the value in the bracket.

If $l_B(\beta'') \ge (4/3) \cdot r$ or $\beta$ connects two points away from each $U(p_i,r)$, then
$l_B(\beta'' \cap K_r) \ge (1/2) l_B(\beta'')$.
Hence, we get
\begin{align*}
l_{\Moduli_h}\big( F(\beta) \big)
&\ge
l_{\Moduli_h}\left(
F(\beta \cap K_r)
\right)
\ge
\frac{1}{K_7(F,\boldsymbol{\epsilon},r)} \cdot
l_B(\beta \cap K_r) \\
&\ge
\frac{1}{K_7(F,\boldsymbol{\epsilon},r)} \cdot
l_B(\beta'' \cap K_r)
\ge
\frac{1}{2 K_7(F,\boldsymbol{\epsilon},r)} \cdot
l_B(\beta'')
\ge
\frac{1}{2 K_7(F,\boldsymbol{\epsilon},r)} \cdot
l_B(\beta').
\end{align*}
Otherwise, we get $l_B(\beta') \le l_B(\beta'') \le (4/3) \cdot r$.
Thus, we get the desired inequalities.
\end{proof}

\subsection{Shrinking maps are quasi-isometric immersions}
\label{subsection::qii}

\begin{proof}[Proof of Theorem \ref{thmx::shrinking} (a)]
This comes from Proposition \ref{proposition::qie-cusp} and the fact that $B_{\text{cp}}$ has bounded diameter.
\end{proof}

\begin{proof}[Proof of Theorem \ref{thmx::qii}]
By Proposition \ref{proposition::torsion-peripheral-qii},
the holomorphic map $F$ does not satisfy $\qii(F)$ if some peripheral monodromy has finite order.
This theorem comes from Theorem \ref{thmx::shrinking} (a).
\end{proof}

\subsection{Shrinking maps have quasi-isometric immersed lifts}
\label{subsection::qiit}

\begin{proof}[Proof of Theorem \ref{thmx::shrinking} (b)]
Suppose that all peripheral monodromies are of infinite order.
Let $\beta\subset B$ be an arbitrary path joining $b_L$ to $b_R$.

Let $0 < \epsilon_0 \le 2$ be such that $\Sing(F)$ does not intersect with the boundary of the cusp region $U_i(\epsilon_0)$ of the $i$-th cusp point, where $l_B(\partial U_i(\epsilon_0)) = \epsilon_0$, for each $i=1,\ldots, n$.
By Proposition \ref{proposition::bilipschitz-path-Bcp} (a), the intersection $\Sing(F) \cap B_{\text{cp}}(\epsilon_0,\ldots,\epsilon_0) = \{p_1,\ldots,p_k\}$ is finite.
Take $0 < \epsilon^+ \le \epsilon_0$ as in Proposition \ref{proposition::cusp-bi-lipschitz}.
Let $r>0$ be such that the hyperbolic neighbourhodds $U(p_i,r)$ are disjoint and away from each of the cusp regions $U_i(\epsilon_0)$.

We decompose $\beta$ into subarcs as follows.
First, we apply the following procedure repeatedly.
If there exists a subarc $\delta \subset U_i(\epsilon_0)$ of $\beta$ such that
$\delta$ joins two boundary points on $\partial U_i(\epsilon_0)$ and
$\delta \not\subset \bigsqcup\limits_{\epsilon^+ \le \epsilon \le  \epsilon_0} H_{U_i(\epsilon_0),\epsilon}$,
then we remove $\delta$ from $\beta$ and, by Proposition \ref{proposition::cusp-bi-lipschitz}, we have
\begin{equation*}
\frac{1}{2} l_U(\delta)
\ge l_{\Moduli_h}(F(\delta))
\ge \frac{1}{K_5(F,\epsilon_0)} l_U(\delta'),
\end{equation*}
where $\delta'$ is the hyperbolic geodesic rel isotopic to $\delta$.

The resulting $\beta$ is the disjoint union of several arcs.
Except the first and the last subarcs, every subarc $\sigma$ is a path in $B_{\text{cp}}(\epsilon^+,\ldots,\epsilon^+)$
joining $b_L \in \partial U_i(\epsilon_0)$ to $b_R \in \partial U_j(\epsilon_0)$, for some $1\le i,j \le n$.
By Proposition \ref{proposition::bilipschitz-path-Bcp} (c), we get
\begin{equation*}
\frac{1}{2} l_U(\sigma)
\ge l_{\Moduli_h}(F(\sigma))
\ge \frac{1}{K_7(F,\epsilon^+)} l_B(\sigma'),
\end{equation*}
where $\sigma'$ is the hyperbolic geodesic rel isotopic to $\sigma$.

Each of the remaining at most two arcs can be decomposed into at most three subarcs, where each subarc is either within some cusp region $U_i(\epsilon_0)$ or within the compact region $B_{\text{cp}}(\epsilon^+,\ldots,\epsilon^+)$.
By Proposition \ref{proposition::cusp-quasi-isom} and Proposition \ref{proposition::bilipschitz-path-Bcp} (c), each subarc $\tau$ satisfies
\begin{equation*}
\frac{1}{2} l_B(\tau) \ge
l_{\Moduli_h}(F(\tau)) \ge
\min\left\{
\frac{1}{K_3(F,\epsilon_0)},
\frac{1}{K_7(F,\epsilon^+)}
\right\} l_B(\tau')
- \max\left\{
K_4(F,\epsilon_0), K_8(F,\epsilon^+)
\right\},
\end{equation*}
where $\tau'$ is the hyperbolic geodesic rel isotopic to $\tau$.

We conclude the proof with
\begin{align*}
l_{\mathcal{M}_h}(F(\beta))
=&
\sum_{\delta} l_{\mathcal{M}_h}(F(\delta))
+
\sum_{\sigma} l_{\mathcal{M}_h}(F(\sigma))
+
\sum_{\tau} l_{\mathcal{M}_h}(F(\tau)) \\
\ge&
\frac{1}{K_5(F,\epsilon_0)}
\sum_{\delta} l_B(\delta')
+
\frac{1}{K_7(F,\epsilon^+)}
\sum_{\sigma} l_B(\sigma') \\
&+
\sum_{\tau} \left(
\min\left\{
\frac{1}{K_3(F,\epsilon_0)},
\frac{1}{K_7(F,\epsilon^+)}
\right\} l_B(\tau')
- \max\left\{
K_4(F,\epsilon_0), K_8(F,\epsilon^+)
\right\}
\right) \\
\ge&
\min\left\{
\frac{1}{K_3(F,\epsilon_0)},
\frac{1}{K_5(F,\epsilon_0)},
\frac{1}{K_7(F,\epsilon^+)}
\right\}
\left(
\sum_\delta l_B(\delta') +
\sum_\sigma l_B(\sigma') +
\sum_\tau l_B(\tau')
\right) \\
&-
6 \max\left\{
K_4(F,\epsilon_0), K_8(F,\epsilon^+)
\right\} \\
\ge&
\min\left\{
\frac{1}{K_3(F,\epsilon_0)},
\frac{1}{K_5(F,\epsilon_0)},
\frac{1}{K_7(F,\epsilon^+)}
\right\}
l_B(\beta')
-
6 \max\left\{
K_4(F,\epsilon_0), K_8(F,\epsilon^+)
\right\}.
\end{align*}
\end{proof}

\begin{proof}[Proof of Theorem \ref{thmx::qiit}]
By Corollary \ref{corollary::torsion-peripheral-qiit}, the holomorphic map $F$ does not satisfy $\qiit(F)$ if some peripheral monodromy has finite order.
This theorem comes from Theorem \ref{thmx::shrinking} (b).
\end{proof}

\begin{remark}
The implication $\qiit(F) \Rightarrow \qii(F)$ gives an alternative proof of Theorem \ref{thmx::qii} using Theorem \ref{thmx::qiit}.
\end{remark}

\subsection{Shrinking maps have quasi-isometric embedded lifts}
\label{subsection::qiet}

\begin{proof}[Proof of Theorem \ref{thmx::shrinking} (c)]

Suppose that $F$ satisfies $\qiet(F)$, namely, there exist $\lambda \ge 1$, $\epsilon \ge 0$ such that
\begin{equation*}
\frac{1}{2} d_{\Hyperbolic^2}(\widetilde{b_1},\widetilde{b_2})
\ge
\dtei{h}(\widetilde{F}(\widetilde{b_1}),\widetilde{F}(\widetilde{b_2}))
\ge
\frac{1}{\lambda} d_{\Hyperbolic^2}(\widetilde{b_1},\widetilde{b_2}) - \epsilon
\end{equation*}
for every $\widetilde{b_1}$ and $\widetilde{b_2} \in \Hyperbolic^2$.
Fix $\widetilde{b_0} \in \Hyperbolic^2$, which is a lift of $b_0 \in B$.
Set $\widetilde{q_0} = \widetilde{F}(\widetilde{b_0})$.
Consider the orbits $\pi_1(B,t) \cdot \widetilde{b_0}$ and $F_*(\pi_1(B,t)) \cdot \widetilde{q_0}$.
Therefore, for every $\phi_1, \phi_2 \in \pi_1(B,t)$, we get
\begin{equation*}
d_{\pi_1(B,t) \curvearrowright \Hyperbolic^2}(\phi_1, \phi_2) = d_{\Hyperbolic^2}(\phi_1 \cdot \widetilde{b_0}, \phi_2 \cdot \widetilde{b_0})
\end{equation*}
and
\begin{equation*}
d_{\Mod_h \curvearrowright \Teich_h}(F_* \phi_1, F_* \phi_2) =
\dtei{h}(\widetilde{F}(\phi_1 \cdot \widetilde{b_0}), \widetilde{F}(\phi_2 \cdot \widetilde{b_0})) = 
\dtei{h}( F_* \phi_1 \cdot \widetilde{q_0}, F_* \phi_2 \cdot \widetilde{q_0}).
\end{equation*}
Hence $F_*$ is a quasi-isometric embedding.

On the other hand, suppose that $F_*:(\pi_1(B,t), d_{\pi_1(B,t) \curvearrowright \Hyperbolic^2}) \rightarrow (\Mod_h, d_{\Mod_h \curvearrowright \Teich_h})$
is a quasi-isometric embedding with the parameters $(\lambda,\epsilon)$, namely, for every lifts $\widetilde{t_L}$ and $\widetilde{t_R}$ of $t$, we get
\begin{equation*}
\frac{1}{2} d_{\Hyperbolic^2}(\widetilde{t_L},\widetilde{t_R}) \ge
\dtei{h}(\widetilde{F}(\widetilde{t_L}), \widetilde{F}(\widetilde{t_R})) \ge
\frac{1}{\lambda} d_{\Hyperbolic^2}(\widetilde{t_L},\widetilde{t_R}) - \epsilon.
\end{equation*}
Let $\widetilde{\beta} \subset \Hyperbolic^2$ be a geodesic segment joining $\widetilde{b_L}$ to $\widetilde{b_R}$, where $\widetilde{b_L}$, $\widetilde{b_R}$ are lifts of $b_L$, $b_R \in B$, respectively.
Set $\widetilde{q_L} = \widetilde{F}(\widetilde{b_L})$,
$\widetilde{q_R} = \widetilde{F}(\widetilde{b_R})$.
We use $\mathcal{U} \subset \Hyperbolic^2$ to denote the union of all the lifts of standard cusps regions of $B$.
There are then four cases to consider.

\begin{enumerate}[label=(\roman*).]
\item Suppose that $\widetilde{b_L}, \widetilde{b_R} \in \overline{ \Hyperbolic^2 \setminus \mathcal{U} }$.

There exist lifs $\widetilde{t_L}$ and $\widetilde{t_R}$ of $t \in B$ such that
$d_{\Hyperbolic^2}(\widetilde{t_L}, \widetilde{b_L}) \le \diam_B(B_{\text{cp}})$
and
$d_{\Hyperbolic^2}(\widetilde{t_R}, \widetilde{b_R}) \le \diam_B(B_{\text{cp}})$,
where $B_{\text{cp}} \subset B$ is obtained from $B$ by removing all the standard cusp regions.
Therefore, we get
\begin{align*}
\dtei{h}(\widetilde{F}(\widetilde{b_L}),\widetilde{F}(\widetilde{b_R}))
&\ge
\dtei{h}(\widetilde{F}(\widetilde{t_L}),\widetilde{F}(\widetilde{t_R}))
- \dtei{h}(\widetilde{F}(\widetilde{t_L}),\widetilde{F}(\widetilde{b_L}))
- \dtei{h}(\widetilde{F}(\widetilde{t_R}),\widetilde{F}(\widetilde{b_R})) \\
&\ge
\frac{1}{\lambda} d_{\Hyperbolic^2}(\widetilde{t_L},\widetilde{t_R}) - \epsilon
- \frac{1}{2}\left(
d_{\Hyperbolic^2}(\widetilde{t_L},\widetilde{b_L})
+ d_{\Hyperbolic^2}(\widetilde{t_R},\widetilde{b_R})
\right) \\
&\ge
\frac{1}{\lambda} d_{\Hyperbolic^2}(\widetilde{b_L},\widetilde{b_R}) - 2 \diam_B(B_{\text{cp}}) - \epsilon - \diam_B(B_{\text{cp}}).
\end{align*}

To simplify the notation, we use $(\lambda_{(\RNum{1})}, \epsilon_{(\RNum{1})})$ to denote the parameters in the above desired quasi-isometric embedding.

\item Suppose that there exists a lift $\widetilde{U}\subset \Hyperbolic^2$ of some standard cusp region such that $\widetilde{b_L}, \widetilde{b_R} \in \widetilde{U}$.

By Proposition \ref{proposition::cusp-quasi-isom}, we get
\begin{equation*}
\dtei{h}(\widetilde{q_L},\widetilde{q_R}) \ge \frac{1}{K_1(F)} d_{\Hyperbolic^2}(\widetilde{b_L}, \widetilde{b_R}) - K_2(F).
\end{equation*}

We use $(\lambda_{(\RNum{2})}, \epsilon_{(\RNum{2})})$ to denote the parameters in the above desired quasi-isometric embedding.

\item Suppose that there exists a lift $\widetilde{U} \subset \Hyperbolic^2$ of some standard cusp region such that $\widetilde{b_L} \in \widetilde{U}$ but $\widetilde{b_R} \in \overline{\Hyperbolic^2 \setminus \mathcal{U}}$.

Consider the intersection $\widetilde{b_M} \coloneqq \widetilde{\beta} \cap \partial \widetilde{U}$,
which satisfies
\begin{equation*}
d_{\Hyperbolic^2}(\widetilde{b_L},\widetilde{b_R}) =
d_{\Hyperbolic^2}(\widetilde{b_L},\widetilde{b_M}) +
d_{\Hyperbolic^2}(\widetilde{b_M},\widetilde{b_R}).
\end{equation*}
Set $\widetilde{q_M}=\widetilde{F}(\widetilde{b_M})$.

If $d_{\Hyperbolic^2}(\widetilde{b_L},\widetilde{b_M}) \le \frac{1}{2\lambda_{(\RNum{1}})} d_{\Hyperbolic^2}(\widetilde{b_M},\widetilde{b_R})$,
then
\begin{align*}
d_\Teich(\widetilde{q_L},\widetilde{q_R})
&\ge
d_\Teich(\widetilde{q_M},\widetilde{q_R}) - d_\Teich(\widetilde{q_L},\widetilde{q_M})
\ge
\frac{1}{\lambda_{(\RNum{1})}} d_{\Hyperbolic^2}(\widetilde{b_M},\widetilde{b_R}) - \epsilon_{(\RNum{1})} - d_{\Hyperbolic^2}(\widetilde{b_L},\widetilde{b_M}) \\
&\ge \frac{1}{2\lambda_{(\RNum{1})}} d_{\Hyperbolic^2}(\widetilde{b_M},\widetilde{b_R}) - \epsilon_{(\RNum{1})}
\ge \frac{1}{2\lambda_{(\RNum{1})}}
\frac{1}{1+\frac{1}{2 \lambda_{(\RNum{1})}}}
d_{\Hyperbolic^2}(\widetilde{b_L},\widetilde{b_R}) - \epsilon_{(\RNum{1})}.
\end{align*}

If $d_{\Hyperbolic^2}(\widetilde{b_M},\widetilde{b_R}) \le \frac{1}{2\lambda_{(\RNum{2}})} d_{\Hyperbolic^2}(\widetilde{b_L},\widetilde{b_M})$,
then
\begin{align*}
d_\Teich(\widetilde{q_L},\widetilde{q_R})
&\ge
d_\Teich(\widetilde{q_L},\widetilde{q_M}) - d_\Teich(\widetilde{q_M},\widetilde{q_R})
\ge
\frac{1}{\lambda_{(\RNum{2})}} d_{\Hyperbolic^2}(\widetilde{b_L},\widetilde{b_M}) - \epsilon_{(\RNum{2})} - d_{\Hyperbolic^2}(\widetilde{b_M},\widetilde{b_R}) \\
&\ge \frac{1}{2\lambda_{(\RNum{2})}} d_{\Hyperbolic^2}(\widetilde{b_L},\widetilde{b_M}) - \epsilon_{(\RNum{2})}
\ge \frac{1}{2\lambda_{(\RNum{2})}}
\frac{1}{1+\frac{1}{2 \lambda_{(\RNum{2})}}}
d_{\Hyperbolic^2}(\widetilde{b_L},\widetilde{b_R}) - \epsilon_{(\RNum{2})}.
\end{align*}

From now on, we suppose that
\begin{equation*}
d_{\Hyperbolic^2}(\widetilde{b_L},\widetilde{b_M}) \ge \frac{1}{2\lambda_{(\RNum{1})}} d_{\Hyperbolic^2}(\widetilde{b_M},\widetilde{b_R})
\quad
\text{and}
\quad
d_{\Hyperbolic^2}(\widetilde{b_M},\widetilde{b_R}) \ge \frac{1}{2\lambda_{(\RNum{2})}} d_{\Hyperbolic^2}(\widetilde{b_L},\widetilde{b_M}).
\end{equation*}
Let $\widetilde{b_H} \in \partial \widetilde{U}$ be such that $d_{\Hyperbolic^2}(\widetilde{b_L},\widetilde{b_H}) = \log (2 / \epsilon_L)$, where $\epsilon_L = \epsilon_U(b_L)$.
We get
\begin{align*}
d_\Teich(\widetilde{q_L},\widetilde{q_R})
&\ge
d_\Teich(\widetilde{q_H},\widetilde{q_R}) - d_\Teich(\widetilde{q_L},\widetilde{q_H})
\ge
\frac{1}{\lambda_{(\RNum{1})}}
d_{\Hyperbolic^2}(\widetilde{b_H},\widetilde{b_R}) - \epsilon_{(\RNum{1})}
- \log{\frac{2}{\epsilon_L}} \\
&\ge
\frac{1}{\lambda_{(\RNum{1})}}
\left(
d_{\Hyperbolic^2}(\widetilde{b_L},\widetilde{b_R}) - d_{\Hyperbolic^2}(\widetilde{b_L},\widetilde{b_H})
\right)
- \epsilon_{(\RNum{1})}
- \log{\frac{2}{\epsilon_L}} \\
&=
\frac{1}{\lambda_{(\RNum{1})}}
d_{\Hyperbolic^2}(\widetilde{b_L},\widetilde{b_R})
- \epsilon_{(\RNum{1})}
- \left(1+\frac{1}{\lambda_{(\RNum{1})}}\right)
\log{\frac{2}{\epsilon_L}} \\
&\ge
\left(1+\frac{1}{2\lambda_{(\RNum{2})}} \right)
\frac{1}{\lambda_{(\RNum{1})}}
\cdot d_{\Hyperbolic^2}(\widetilde{b_L},\widetilde{b_M})
- \epsilon_{(\RNum{1})}
- \left(1+\frac{1}{\lambda_{(\RNum{1})}}\right)
\log{\frac{2}{\epsilon_L}}.
\end{align*}
If $d_{\Hyperbolic^2}(\widetilde{b_L},\widetilde{b_M}) \ge K' \cdot \log{\frac{2}{\epsilon_L}}$,
where $K' \coloneqq
\left(
1 + \frac{1}{\lambda_{(\RNum{1})}}
\right) /
\left(
\frac{1}{2}
\left(
1 + \frac{1}{2 \lambda_{(\RNum{2})}}
\right)
\frac{1}{\lambda_{(\RNum{2})}}
\right)$,
then
\begin{align*}
d_\Teich(\widetilde{q_L},\widetilde{q_R})
&\ge
\frac{1}{2}
\left(
1 + \frac{1}{2\lambda_{(\RNum{2})}}
\right)
\frac{1}{\lambda_{(\RNum{1})}}
d_{\Hyperbolic^2}(\widetilde{b_L},\widetilde{b_M}) - \epsilon_{(\RNum{1})} \\
&\ge
\frac{1}{2}
\left(
1 + \frac{1}{2\lambda_{(\RNum{2})}}
\right)
\frac{1}{\lambda_{(\RNum{1})}}
\frac{1}{1+2\lambda_{(\RNum{1})}}
d_{\Hyperbolic^2}(\widetilde{b_L},\widetilde{b_M}) - \epsilon_{(\RNum{1})}.
\end{align*}

Otherwise, consider the case when $d_{\Hyperbolic^2}(\widetilde{b_L},\widetilde{b_M}) \le K' \cdot \log{\frac{2}{\epsilon_L}}$.
As in Lemma \ref{lemma::translation_length} (b), we suppose that the peripheral monodromy corresponding to $\widetilde{U}$ has the power $\phi^\mu = T_{\alpha_1}^{r_1}\circ \cdots \circ T_{\alpha_m}^{r_m}$.
By Wolpert's inequality and Proposition \ref{proposition::l_epsilon},
\begin{align*}
\dtei{h}(\widetilde{q_L},\widetilde{q_R})
&\ge
\frac{1}{2} \log \frac{L_{\alpha_i}(\widetilde{q_R})}{L_{\alpha_i}(\widetilde{q_L})} \\
&\ge
\frac{1}{2} \log
\frac{
\max\limits_{\widetilde{b}\in \Hyperbolic^2 \setminus \mathcal{U}}
\left\{ \sys( \widetilde{F}(\widetilde{b}) )
\right\}
}{ K^+ \cdot \epsilon_L }
=
\frac{1}{2} \log \frac{2}{\epsilon_L}
+
\frac{1}{2} \log
\frac{
\max\limits_{\widetilde{b}\in \Hyperbolic^2 \setminus \mathcal{U}}
\left\{ \sys( \widetilde{F}(\widetilde{b}) ) \right\}
}{ 2 K^+ } \\
&\ge
\frac{1}{2K'} d_{\Hyperbolic^2}(\widetilde{b_L},\widetilde{b_M})
+
\frac{1}{2} \log
\frac{\max\limits_{\widetilde{b}\in \Hyperbolic^2 \setminus \mathcal{U}}
\left\{ \sys( \widetilde{F}(\widetilde{b}) ) \right\}
}{ 2 K^+ } \\
&\ge
\frac{1}{2K'(1+2\lambda_{(\RNum{1})})} d_{\Hyperbolic^2}(\widetilde{b_L},\widetilde{b_M})
+
\frac{1}{2} \log
\frac{\max\limits_{\widetilde{b}\in \Hyperbolic^2 \setminus \mathcal{U}}
\left\{ \sys( \widetilde{F}(\widetilde{b}) ) \right\}
}{ 2 K^+ }.
\end{align*}

We use $(\lambda_{(\RNum{3})}, \epsilon_{(\RNum{3})})$ to denote the parameters in the desired quasi-isometric embedding.

\item Suppose that there exist distinct lifts $\widetilde{U_L}$ and $\widetilde{U_R}$ of standard cusp regions such that
$\widetilde{b_L} \in \widetilde{U_L}$ and
$\widetilde{b_R} \in \widetilde{U_R}$.

Consider the intersections $\widetilde{b_G} \coloneqq \widetilde{\beta} \cap \partial \widetilde{U_L}$
and
$\widetilde{b_D} \coloneqq \widetilde{\beta} \cap \partial \widetilde{U_R}$,
which satisfy
\begin{equation*}
\dhyp(\widetilde{b_L},\widetilde{b_R}) = 
\dhyp(\widetilde{b_L},\widetilde{b_G}) +
\dhyp(\widetilde{b_G},\widetilde{b_D}) +
\dhyp(\widetilde{b_D},\widetilde{b_R}).
\end{equation*}
Set $\widetilde{q_G} = \widetilde{F}(\widetilde{b_G})$
and $\widetilde{q_D} = \widetilde{F}(\widetilde{b_D})$.

As in the argument for the case (\RNum{3}), we suppose that
\begin{align*}
\dhyp(\widetilde{b_L},\widetilde{b_G}) \ge \frac{1}{2 \lambda_{(\RNum{3})}} \dhyp(\widetilde{b_G},\widetilde{b_R})
\quad
\text{and}
\quad
\dhyp(\widetilde{b_D},\widetilde{b_R}) \ge \frac{1}{2 \lambda_{(\RNum{3})}} \dhyp(\widetilde{b_L},\widetilde{b_D}), \\
\dhyp(\widetilde{b_L},\widetilde{b_D}) \ge \frac{1}{2 \lambda_{(\RNum{2})}} \dhyp(\widetilde{b_D},\widetilde{b_R})
\quad
\text{and}
\quad
\dhyp(\widetilde{b_G},\widetilde{b_R}) \ge \frac{1}{2 \lambda_{(\RNum{2})}} \dhyp(\widetilde{b_L},\widetilde{b_G}).
\end{align*}

Let $\widetilde{b_H} \in \partial \widetilde{U_L}$ be such that $\dhyp(\widetilde{b_L},\widetilde{b_H}) = \log (2/\epsilon_L)$. 
As in the case (\RNum{3}), we get
\begin{equation*}
\dtei{h}(\widetilde{q_L},\widetilde{q_R}) \ge
\frac{1}{\lambda_{(\RNum{3})}}
\left(
1+\frac{1}{2\lambda_{(\RNum{2})}}
\right)
\dhyp(\widetilde{b_L},\widetilde{b_G})
- \epsilon_{(\RNum{3})}
- \left(
1+\frac{1}{\lambda_{(\RNum{3})}}
\log \frac{2}{\epsilon_L}
\right).
\end{equation*}
If $\dhyp(\widetilde{b_L},\widetilde{b_G}) \ge K' \cdot \log \frac{2}{\epsilon_L}$, where
$K' \coloneqq
\left(
1 + \frac{1}{\lambda_{(\RNum{3})}}
\right) /
\left(
\frac{1}{2}
\left(
1 + \frac{1}{2 \lambda_{(\RNum{2})}}
\right)
\frac{1}{\lambda_{(\RNum{3})}}
\right)$,
then we get the desired quasi-isometric embedding.

Otherwise, consider the case when
$\dhyp(\widetilde{b_L},\widetilde{b_G}) \le K' \cdot \log \frac{2}{\epsilon_L}$.
As in Lemma \ref{lemma::translation_length} (b), we suppose that the peripheral monodromies along $\widetilde{\beta}$ corresponding to $\widetilde{U_L}$ and $\widetilde{U_R}$ have the power
$\phi_1^{\mu_1} = \prod T_{\alpha_i}^{r_i}$ and
$\phi_2^{\mu_2} = \prod T_{\gamma_j}^{s_j}$, respectively.
By Lemma \ref{lemma::peripherals_intersecting},
the multicurves $\boldsymbol{\alpha} = \{\alpha_i\}$ intersects with $\boldsymbol{\gamma} = \{\gamma_j\}$.
Suppose that $\Int(\alpha_i,\gamma_j) \ge 1$, for some $i$ and $j$.
By the collar inequality and Proposition \ref{proposition::l_epsilon}, we get
\begin{equation*}
L_{\alpha_i}(\widetilde{q_R}) \ge 2 \arcsinh \frac{1}{\sinh (K^+ \cdot \epsilon_R / 2)}
\ge 2 \arcsinh \frac{1}{\sinh K^+}.
\end{equation*}
Therefore, by Wolpert's inequality, we have
\begin{align*}
\dtei{h}(\widetilde{q_L},\widetilde{q_R})
&\ge
\frac{1}{2} \log \frac{2}{\epsilon_L} + \frac{1}{2} \log \arcsinh \frac{1}{\sinh K^+} \\
&\ge 
\frac{1}{2K^+} \dhyp(\widetilde{b_L},\widetilde{b_G}) + \frac{1}{2} \log \arcsinh \frac{1}{\sinh K^+} \\
&\ge
\frac{1}{2K^+ \cdot (1+2\lambda_{(\RNum{3})})} \dhyp(\widetilde{b_L},\widetilde{b_R}) + \frac{1}{2} \log \arcsinh \frac{1}{\sinh K^+}.
\end{align*}
\end{enumerate}

This completes the proof of Theorem \ref{thmx::shrinking}.
\end{proof}

\sloppy
\printbibliography[
    heading=bibintoc,
    title={References}
]

\end{document}